\documentclass[10pt, a4paper,intlimits]{amsart}

\usepackage[theorem,hyperref,comment,enumitem,Rn,frak]{paper_diening}

\usepackage{mathrsfs}
\usepackage{bbm}

\usepackage{dsfont}

\usepackage{geometry}
\usepackage{graphicx}

\providecommand{\opA}{\mathbb{A}}
\providecommand{\A}{\opA}

\providecommand{\RK}{\setR^K}

 \newtheorem{assumption}[theorem]{Assumption}

\providecommand{\R}{\mathbb{R}}
\providecommand{\dif}{\operatorname{d}\!}

\providecommand{\bv}{\setBV}
\providecommand{\bd}{\operatorname{BD}}
\providecommand{\BVA}{\setBV^\opA}
\providecommand{\bvA}{\BVA}
\providecommand{\WA}{W^{\opA,1}}

\providecommand{\hold}{\operatorname{C}}
\providecommand{\sobo}{W}
\providecommand{\lebe}{L}
\providecommand{\tracenew}{\operatorname{\widetilde{\trace}}}
\providecommand{\locc}{\operatorname{loc}}
\providecommand{\ball}{\operatorname{B}}

\providecommand{\spt}{\operatorname{spt}}

\providecommand{\wstar}{\weakastto}
\providecommand{\strict}{\strictto}
\providecommand{\strictto}{\stackrel{s}{\rightarrow}}
\providecommand{\areato}{\stackrel{\langle\cdot\rangle}{\rightarrow}}

\providecommand{\di}{\operatorname{div}}
\providecommand{\ball}{\operatorname{B}}

\renewcommand{\Re}{\operatorname{Re}}
\renewcommand{\Im}{\operatorname{Im}}
\newcommand{\ld}{\operatorname{LD}}

\newcommand{\mres}{\!\mathbin{\vrule height 1.6ex depth 0pt width
0.13ex\vrule height 0.13ex depth 0pt width 1.1ex}\!}

\providecommand{\CA}{\ensuremath{\mathscr{C}(\opA)}}
\providecommand{\RA}{\ensuremath{\mathscr{R}(\opA)}}

\numberwithin{equation}{section}
\numberwithin{theorem}{section}

\begin{document}

\title[Traces of $\bvA$--maps]{On the Trace Operator for Functions of Bounded
  $\mathbb{A}$--Variation}

\author[D.~Breit]{Dominic Breit}
\address{Dominic Breit, Department of Mathematics, Heriot-Watt
    University, Riccarton Edinburgh EH14 4AS, UK}
\email{d.breit@hw.ac.uk}

\author[L.~Diening]{Lars Diening}
\address{Lars Diening, Universität Bielefeld,
Fakultät für Mathematik,
Postfach 10 01 31,
D-33501 Bielefeld}
\email{lars.diening@uni-bielefeld.de}

\author[F.~Gmeineder]{Franz Gmeineder}
\address{Franz Gmeineder, University of Bonn, Department of Applied Mathematics, Endenicher Allee 60, 53115 Bonn, Germany}
\email{gmeineder@maths.ox.ac.uk}

\maketitle
\begin{abstract}
In this paper, we consider the space $\BVA(\Omega)$ of functions of bounded $\A$-variation. For a given first order linear homogeneous differential operator with constant coefficients $\A$, this is the space of $L^1$--functions $u:\Omega\rightarrow\mathbb R^N$ such that the distributional differential expression $\A u$ is a finite (vectorial) Radon measure. We show that for Lipschitz domains $\Omega\subset\R^{n}$, $\BVA(\Omega)$--functions have an $L^1(\partial\Omega)$--trace if and only if $\A$ is $\mathbb C$-elliptic (or, equivalently, if the kernel of $\A$ is finite dimensional). The existence of an $L^1(\partial\Omega)$--trace was previously only known for the special cases that $\A u$ coincides either with the full or the symmetric gradient of the function $u$ (and hence covered the special cases $\bv$ or $\bd$). 
As a main novelty, we do not use the fundamental theorem of calculus to construct the trace operator (an approach which is only available in the $\bv$- and $\bd$-setting) but rather compare projections onto the nullspace as we approach the boundary. 
As a sample application, we study the Dirichlet problem for quasiconvex variational functionals with linear growth depending on $\A u$. 
\end{abstract}


\section{Introduction}
\label{sec:introduction}
\subsection{Aim and Scope}
Let $\Omega$ be an open, bounded Lipschitz domain in $\R^{n}$ and let $1\leq p <\infty$. A key tool in the study of partial differential equations is the assignment of boundary values to elements $u\in\sobo^{1,p}(\Omega;\R^{N})$, often being the first step towards well-posedness results for such equations. In this respect, it is a well-established fact (cf.~\cite{Mazya}) that if $1<p<\infty$, then there exists a surjective, bounded linear trace embedding operator 
\begin{align}\label{eq:1}
\trace:\sobo^{1,p}(\Omega;\R^{N}) \hookrightarrow\sobo^{1-1/p,p}(\partial\Omega;\R^{N})
\end{align}
which satisfies $\trace(u)=u|_{\partial\Omega}$ for $u\in C(\overline\Omega;\R^N)\cap\sobo^{1,p}(\Omega;\R^{N})$. If $p=1$ instead, a result due to \textsc{Gagliardo} \cite{Gal57} asserts that there exists a surjective, bounded linear trace embedding operator
\begin{align}
\label{eq:2}
\trace:\sobo^{1,1}(\Omega;\R^{N}) \hookrightarrow\lebe^{1}(\partial\Omega;\R^{N}).
\end{align}
The same holds true when $\sobo^{1,1}(\Omega;\R^{N})$ is replaced by $\bv(\Omega;\R^{N})$, the $\R^{N}$-valued functions of bounded variation on $\Omega$. Both boundary trace embeddings \eqref{eq:1}, \eqref{eq:2} and the corresponding variant for $\bv$
hinge on inequalities 
\begin{align}\label{eq:traceinequalitiesgeneral}
\begin{split}
&\|u\|_{\sobo^{1-\frac{1}{p},p}(\partial\Omega;\R^{N})}\leq C (\|u\|_{\lebe^{p}(\Omega;\R^{N})}+\|Du\|_{\lebe^{p}(\Omega;\R^{N\times n})})\\
& \|u\|_{\lebe^{1}(\partial\Omega;\R^{N})}\leq C (\|u\|_{\lebe^{1}(\Omega;\R^{N})}+\|Du\|_{\lebe^{1}(\Omega;\R^{N\times n})})
\end{split}
\end{align}
if $1<p<\infty$ or $p=1$, respectively, to be satisfied for all $u\in \hold(\overline{\Omega};\R^{N})\cap\sobo^{1,p}(\Omega;\R^{N})$. These estimates in turn are obtained as a consequence of the fundamental theorem of calculus in conjunction with a smooth approximation argument.

As one of the fundamental achievements of 20th century harmonic analysis, \textsc{Calder\'{o}n \& Zygmund} \cite{MR18:894a} and \textsc{Mihlin} \cite{Mih56} established that in a wealth of inequalities, the \emph{full gradient} can be replaced by weaker quantities only involving certain combinations of derivatives. Precisely, let $\A$ be a
constant--coefficient, linear, homogeneous differential operator from
$\RN$ to $\RK$, i.e., there exist fixed linear maps
$\A_\alpha\colon \RN\to \RK$ with
\begin{align}
  \label{eq:form}
  \A=\sum_{\alpha=1}^n \A_\alpha\partial_\alpha.
\end{align}
Then
for each $1<p<\infty$ there exists $c=c(p,n,\A)>0$ such that there holds
\begin{align}\label{eq:calderonzygmund}
\|D u\|_{\lebe^{p}(\R^{n};\R^{N\times n})}\leq c\|\A u\|_{\lebe^{p}(\R^{n};\R^{K})}\qquad\text{for all}\; u\in\hold_{c}^{\infty}(\R^{n};\R^{N})
\end{align} 
if and only if $\A$ is \emph{elliptic}. Here we say that $\A$ is elliptic if and only if for each $\xi=(\xi_{1},...,\xi_{n})\in\R^{n}\setminus\{0\}$ the \emph{symbol map} $\A[\xi]:=\sum_{\alpha}\xi_{\alpha}\A_{\alpha}\colon \R^{N}\to\R^{K}$ is an injective linear map. 
 A special instance of \eqref{eq:calderonzygmund} is the case of the symmetric gradient operator $\mathcal{E}u:=\frac{1}{2}(Du+D^{\top}\!u)$ acting on maps $u\colon\R^{n}\to\R^{n}$ (here $N=n\geq 2$ and $K=n^{2}$, identifying $\R^{n\times n}\cong \R^{n^{2}}$). In this situation, \eqref{eq:calderonzygmund} gives the usual Korn inequalities which play a pivotal role in elasticity or fluid mechanics; see \cite{FuchsSeregin} for a
comprehensive overview. 

Singular integrals or Fourier multiplier operators 
in general are not bounded on $\lebe^{1}$. Thus one expects the exponent range $1<p<\infty$ for \eqref{eq:calderonzygmund} to hold to be optimal for general elliptic operators $\A$. This is in fact true and manifested by \textsc{Ornstein}'s celebrated Non-Inequality, stating the impossibility of non-trivial $\lebe^{1}$-estimates: 
\vspace{0.25cm}

\textbf{Theorem} (Ornstein (\cite{Or})). \emph{Let $\A$ and $\mathbb{B}$ be two constant-coefficient first order, linear homogeneous differential operators from on $\R^{n}$ from $\R^{N}$ to $\R^{K}$ and from $\R^{N}$ to $\R$, respectively. Suppose that there exists a constant $c>0$ such that 
\begin{align*}
\|\mathbb{B}u\|_{\lebe^{1}(\R^{n})} \leq c\|\A u\|_{\lebe^{1}(\R^{n};\R^{K})}\qquad\text{for all}\;u\in\hold_{c}^{\infty}(\R^{n};\R^{N}). 
\end{align*}
Then there exists $T\in\mathscr{L}(\R^{N};\R)$ such that $\mathbb{B}=T\circ\mathbb{A}$.}
\vspace{0.25cm}

This negative result -- which faces contributions to date, see \cite{CFM,KiKr16} -- immediately yields that if $p=1$, inequalities that involve the full gradients $Du$ do not necessarily generalise to those involving only $\A u$. On the other hand, by \cite{ts} it is known for the special case of $\A$ being the symmetric gradient operator that \eqref{eq:traceinequalitiesgeneral}(b) remains valid indeed for $p=1$ when $D$ is replaced by $\mathcal{E}$. However, the method employed in \cite{ts,Baba} to arrive at this result is very specific to the symmetric gradient operator and its structural properties: Again based on the fundamental theorem of calculus, $\mathcal{E}u$ then allows to
control a cone of line integrals emanating from the boundary, leading to the desired trace inequality. In particular, it is far from clear whether and if so, how, trace inequalities of the form \eqref{eq:traceinequalitiesgeneral} can be established for $p=1$ and $D$ being replaced by differential operators $\A$ of the form \eqref{eq:form}. As we shall see below in Section~\ref{sec:mainresults}, even for general elliptic operators $\A$ the corresponding analogues of \eqref{eq:traceinequalitiesgeneral} break down and hence the method employed for the symmetric gradient cannot easily generalise.  

This leads us to the following \textbf{classification problem}: \emph{Classify all differential operators of the form \eqref{eq:form} such that for any open and bounded Lipschitz domain $\Omega\subset\R^{n}$ there exists a constant $c>0$ such that 
\begin{align}\label{eq:keytraceinequality}
\|u\|_{\lebe^{1}(\partial\Omega;\R^{N})}\leq c (\|u\|_{\lebe^{1}(\Omega;\R^{N})}+\|\A u\|_{\lebe^{1}(\Omega;\R^{K})})
\end{align}
holds for all $u\in\hold(\overline{\Omega};\R^{N})\cap\hold^{1}(\Omega;\R^{N})$.
} The overall objective of the present paper is to solve this classification problem. Before we pass on to the precise description of our results -- in particular, Theorem~\ref{thm:main1} -- we briefly pause and connect this theme to other results available in the literature first. 
\subsection{Contextualisation and Function Spaces}
The quest for classifying differential operators $\A$  of the form \eqref{eq:form} such that well-known inequalities generalise to the $\A$-framework for $p=1$ has come up rather recently. Building on the foundational work of \textsc{Bourgain \& Brezis} \cite{BoBr0,BoBr1,BoBr2}, \textsc{Van Schaftingen} \cite{Van13} characterised all operators $\A$ of the form \eqref{eq:form} for which a Sobolev-type inequality 
\begin{align}\label{eq:VSineq1}
\|u\|_{\lebe^{\frac{n}{n-1}}(\R^{n};\RN)}\leq C\|\A u\|_{\lebe^{1}(\R^{n};\RK)}\qquad\text{for all}\;u\in\hold_{c}^{\infty}(\R^{n};\R^{N})
\end{align}
holds. Whereas ellipticity of $\A$ is easily seen to be necessary for \eqref{eq:VSineq1}, it is far from sufficient and needs to be augmented by the so-called \emph{cancellation condition}. Following \cite{Van13}, we call $\A$ \emph{cancelling} if and only if 
\begin{align*}
\bigcap_{\xi\in\R^{n}\setminus\{0\}}\A[\xi](\R^{N})=\{0\}.
\end{align*}
Note that by ellipticity, 
 $u\in\hold_{c}^{\infty}(\R^{n};\R^{N})$ can be represented via $u=k_{\A}*\A u$ where $k_{\A}\colon\R^{n}\setminus\{0\}\to\mathscr{L}(\R^{K};\R^{N})$ satisfies the growth bound $|k_{\A}(y)|\sim |y|^{1-n}$ for $y\in\R^{n}\setminus\{0\}$. Then the fractional integration theorem only implies that the convolution with $k_{\A}$ yields an operator that maps $\lebe^{1}(\R^{n};\R^{K})\to\lebe_{\text{w}}^{\frac{n}{n-1}}(\R^{n};\R^{N})$ boundedly with the weak-$\lebe^{\frac{n}{n-1}}$ space $\lebe_{\text{w}}^{\frac{n}{n-1}}(\R^{n};\R^{N})$, and so \eqref{eq:VSineq1} implies a proper improvement based on the additional cancellation condition.

To unify this theme also in view of \eqref{eq:keytraceinequality}, we wish to interpret the above inequalities in terms of (boundary trace) embeddings and thus introduce function spaces via  
\begin{align*}
  \WA(\Omega) &:=\bigset{ v\in\lebe^{1}(\Omega;\R^{N})\colon\; \A
  u\in L^1(\Omega; \RK)},\\
  \bvA(\Omega) &:=\big\{ v\in\lebe^{1}(\Omega;\R^{N})\colon\; \A
  u\in\mathcal{M}(\Omega;\RK)\big\},
\end{align*}
where $\Omega\subset\R^{n}$ is open, $\A$ is a differential operator of the form \eqref{eq:form} and $\mathcal{M}(\Omega;\RK)$ denotes the $\R^{K}$-valued Radon measure of finite total variation on $\Omega$. These spaces are normed canonically via $\|u\|_{\sobo^{\A,1}}=\|u\|_{\lebe^{1}}+\|\A u\|_{\lebe^{1}}$ (similarly for $\bv^{\A}$ with the obvious modifications); clearly, $\sobo^{\A,1}(\Omega)\subsetneq\bv^{\A}(\Omega)$ and we shall refer to $\bv^{\A}(\Omega)$ as \emph{space of functions of bounded $\A$-variation}. In the literature, only particular instances of spaces $\bv^{\A}$ have been studied in detail, namely for $\A=\nabla$ or $\A=\mathcal{E}$, leading to the spaces $\bv$ or $\bd$ of functions of bounded variation or deformation, respectively. Precisely, we then have $\sobo^{1,1}=\sobo^{\nabla,1}$, $\ld=\sobo^{\mathcal{E},1}$, $\bv=\bv^{\nabla}$, $\bd=\bv^{\mathcal{E}}$, and this paper is the first attempt to characterise the properties of $\bv^{\A}$-maps in terms of the properties of $\A$ in a unifying manner. By this, we also aim to clarify the underlying mechanisms for the corresponding trace inequalities to work in the known cases $\A=D$ and $\A=\mathcal{E}$.

Returning to the classification problem related to \eqref{eq:keytraceinequality}, we conclude this subsection by pointing out that ellipticity in itself cannot yield the required $\lebe^{1}$-trace theory. In fact, consider the operator $\mathcal{E}^{D}u:=\mathcal{E}u-\frac{1}{n}\di(u)E_{n}$ ($E_{n}\in\R^{n\times n}$ being the identity matrix) which is usually referred to as \emph{trace-free symmetric gradient operator}, for $n\geq 2$. This operator enters in a variety of applications, so for instance fluid mechanics or general relativity, cf. \cite{Fei04} and \cite{BI04}. Regardless of $n\geq 2$, $\mathcal{E}^{D}$ is elliptic, see Example \ref{ex:opexamp} \ref{itm:opexamp3}. However, the following example from \cite{FuchsRepin} shows that an $L^1$-trace does not exists if $n=2$. Identifying $\R^{2}\cong\mathbb{C}$, $\ker(\mathcal{E}^{D})$ essentially contains the holomorphic functions. Upon identifying $\R^{2}$ with $\mathbb{C}$ and denoting $\mathbb{D}$ the open unit disc in $\mathbb{C}$, the map $u\colon\mathbb{D}\ni z\mapsto 1/(z-1)\in\mathbb{C}$ even  belongs to $\sobo^{\mathcal{E}^{D},1}(\ball(0,1))$ whereas it is clear that $\|\trace(u)\|_{\lebe^{1}(\partial\!\ball(0,1))}=\infty$. In view of \eqref{eq:keytraceinequality}, our main result, Theorem~\ref{thm:main1} below, will cover the particular case of $\A=\mathcal{E}^{D}$ as a special case and provide a
positive answer for all~$n\geq 3$ and a negative answer for~$n=2$.


\subsection{Main Results}\label{sec:mainresults}

Before we state our main result, we need to provide the definitions of several important properties
of our operator~$\opA$. To begin with, we write the \emph{symbol mapping}
$\opA[\xi]\,:\, \RN \to \RK$ as
\begin{align}
  \label{eq:bilinear-1}
  \opA[\xi]v &:= v \otimes_\opA \xi := \sum_{\alpha=1}^n \xi_\alpha
               \opA_\alpha v,\qquad \xi=(\xi_{1},...,\xi_{n})\in\R^{n},\;v\in\R^{N}.
\end{align}
Moreover, we extend $\A[\xi]\eta=\eta \otimes_{\opA} \xi$ by~\eqref{eq:bilinear-1} also to complex valued
$\xi \in \setC^n$ and~$\eta \in \setC^N$. We strengthen terminology and say that~$\opA$ is \emph{$\setR$-elliptic} if~$\A[\xi]\colon\R^{N}\to\R^{K}$ is
injective for all $\xi\in\R^{n}\setminus\{0\}$ (i.e., $\A$ is elliptic in the above sense), and \emph{$\setC$-elliptic} provided $\opA[\xi]\colon\mathbb{C}^{N}\to\mathbb{C}^{K}$ is injective for
all~$\xi \in \setC^n \setminus \set{0}$ (cf.~Section~\ref{sec:assumpt-diff-oper} for more detail). Finally, we shall say that $\opA$ has \emph{finite dimensional
  nullspace} if the kernel~$N(\opA)$ of~$\opA$ in the distributional
sense is finite dimensional, i.e.
\begin{align}
  \label{eq:FDN}
  \dim(N(\opA))< \infty \quad \text{ with $N(\opA) = \set{v\in\mathcal{D}'
  (\Rn;\RN)\colon\;\opA v\equiv 0}$},
\end{align}
where $\mathcal{D}(\Rn;\RN) = C^\infty_0(\Rn;\RN)$.
We will see later in Theorem~\ref{thm:opA-equiv} that~$\opA$ has a
finite dimensional nullspace if and only if it is~$\setC$-elliptic. It
is also equivalent to the \emph{type~$(C)$ condition} in the sense of~\cite{Kal94}, see
Remark~\ref{rem:Kal94}. However, the notion of~$\setR$-ellipticity is strictly
weaker: For instance, $\mathcal{E}^D$ for $n=2$ is $\setR$-elliptic but not
$\setC$-elliptic, see Example~\ref{ex:opexamp}~\ref{itm:opexamp3}. We are now in position to formulate our main result. 
\begin{theorem}
Let $\A$ be a differential operator of the form \eqref{eq:form}. Then the following are equivalent: 
\begin{enumerate}
\item For all open and bounded Lipschitz domains $\Omega\subset\R^{n}$ there exists a constant $c>0$ such that \eqref{eq:keytraceinequality} holds for all $u\in\hold(\overline{\Omega};\R^{N})\cap\hold^{1}(\Omega;\R^{N})$. 
\item $\A$ is $\mathbb{C}$-elliptic.
\end{enumerate}
\end{theorem}
Whereas necessity of $\mathbb{C}$-ellipticity for \eqref{eq:keytraceinequality} shall be addressed in Theorem~\ref{thm:no-trace} and essentially follows from a construction relying on the properties of the two-dimensional operator $\mathcal{E}^{D}$, the more involved part is the sufficiency. For future reference, we single this out and state in the following more elaborate form; the full statement
can be found in~Theorem~\ref{thm:main1b}:
\begin{theorem}[Trace theorem]\label{thm:main1}
  Let $\A$ be $\setC$-elliptic (or equivalently, $\A$ has finite
  dimensional nullspace). Then there exists a trace
  operator~$\trace\,:\, \BVA(\Omega) \to L^1(\partial \Omega,
  \mathcal{H}^{n-1})$ such that the following holds:
  \begin{enumerate}
  \item  $\trace(u)$ coincides with the classical
    trace for all $u \in \BVA(\Omega) \cap C(\overline{\Omega};\R^{N})$.
  \item  $\trace(u)$ is the unique
    strictly-continuous extension of the classical trace
    on~$\BVA(\Omega) \cap C(\overline{\Omega};\R^{N})$. Especially, $\trace\colon\bv^{\A}(\Omega)\to\lebe^{1}(\partial\Omega;\mathcal{H}^{n-1})$ is continuous for the norm topology on $\bv^{\A}(\Omega)$. 
  \item    $\trace(\WA(\Omega))=\trace(\BVA(\Omega))=L^1(\partial \Omega;
    \mathcal{H}^{n-1})$.
  \end{enumerate}
\end{theorem}
Regarding sufficiency, the core issue is how to replace the use of the fundamental theorem of calculus by that of $\mathbb{C}$-ellipticity. As a main consequence of the latter, we will employ the nullspace of $\mathbb{C}$-elliptic operators being finite dimensional. 
Using local projections
onto the nullspace $N(\opA)$ close to the boundary, we construct
suitable approximations of~$u \in \BVA(\Omega)$ that have classical
traces. The limit of these traces provide us with the trace of~$u$. In
particular, the projections to the finite dimensional nullspace replace the
fundamental theorem of calculus approach as used in~\cite{ts,Baba}.\\
In addition to Theorem~\ref{thm:main1b} we will show in Theorem~\ref{thm:no-trace} and
Remark~\ref{rem:no-trace} that if~$\opA$ is not $\setC$-elliptic, then
in general there is no trace operator from $\BVA(\Omega)$ to
$L^1(\partial \Omega;\mathcal{H}^{n-1})$. In particular, the existence of
$L^1(\partial \Omega;\mathcal{H}^{n-1})$--traces on arbitrary bounded Lipschitz domains $\Omega\subset\R^{n}$ is equivalent to $\setC$-ellipticity of $\A$. This conclusion also identifies the infinite dimensional nullspace of $\A$ as the reason for the failure of the trace embedding of $\sobo^{\mathcal{E}^{D},1}(\Omega)$ into $\lebe^{1}(\partial\Omega;\mathcal H^{n-1})$ for $n=2$ (cp.~ Example~\ref{ex:opexamp} ~\ref{itm:opexamp3}).
As a consequence of Theorem~\ref{thm:main1} we also obtain a version
of the Gau{\ss}-Green theorem, see Theorem~\ref{thm:gauss-green}, and
the gluing theorem, see Corollary~\ref{cor:gluing}. Let us also remark that Theorem~\ref{thm:main1}  includes both the trace theorems for the spaces $\bv$ and $\bd$.

The relation between the condition of $\mathbb{C}$-ellipticity and \textsc{Van Schaftingen}'s elliptic and cancelling condition will be investigated in detail in the follow-up \cite{GmeRai17} to this paper by \textsc{Raita} and the third author; among others, there will be shown that $\mathbb{C}$-ellipticity implies \textsc{Van Schaftingen}'s condition but in general \emph{not} vice versa. In this sense and as might be anticipated, $\lebe^{1}$-boundary traces require a stronger condition on $\A$.

\subsection{Variational problems}\label{sec:varprob}
As a concluding application of the trace theorem from above, we address the Dirichlet problem for linear growth functionals involving operators $\A$. To be precise, we are interested in the
%
%
minimisation of 
functionals of the form
\begin{align}\label{eq:fctl1}
  \mathfrak{F}[u]:=\int_{\Omega}f(x,\A u)\dif x 
\end{align}
over a class of maps $u\colon\Omega\to\R^{N}$ subject to Dirichlet
boundary data $u=u_{0}$ on $\partial\Omega$. Here
$f\colon\Omega\times\R^{N\times n}\to\R_{\geq 0}$ is a given
variational integrand
for which we suppose the linear growth assumption
\begin{align}\label{eq:1growth}
     c_1|z|&\leq f(x,z)\leq c_2|z|+c_3
                       \qquad\text{for
                       all}\;x\in\Omega\;\text{and}\;z\in\R^{N\times n}.
\end{align}
Additionally, we assume that our integrand $f$
is~\emph{$\opA$-quasiconvex}
(in a sense to specified in
Section~\ref{sec:variational-problems}, also see \cite{FonMul99,DacAfree}). Our objective here is to
minimise $\mathfrak{F}$ over the Dirichlet class $u_0 +
\WA_0(\Omega)$, which are the $\WA(\Omega)$-functions whose traces
agree with the given boundary datum~$u_0$.
From the treatment of the Dirichlet problem on $\bv$ (see
\cite{GiaModSou79,AmbFusPal00}) it is clear that the 
functional should be considered on the class of~$\setBV$--maps on a larger
Lipschitz domain~$U$. More precisely, we need to consider the weak*--lower
semi-continuous envelope of~$\mathfrak{F}$ on $\setBV(U)$.
Whereas in the convex situation one can make use of the classical
results due to Reshetnyak \cite{Resh}, the quasiconvex case is
substantially more involved. The sequentially weak*-lower
semicontinuous envelope $\overline{\mathfrak{F}}$ of $\mathfrak{F}$
on~$\setBV(\Omega)$ (so $\A=\nabla$) was characterized in
\cite{AmbDal92,FoMu91}. The corresponding issue for the
symmetric-quasiconvex (so ${\A=\mathcal{E}}$) situation was resolved
in~\cite{Rindler11}.
Invoking the recent outstanding generalisation of Alberti's rank
one-theorem~\cite{DePRin16}, the weak*--lower semicontinuity result
of~\cite{ArrDePRin17} and the area-strict continuity
of~\cite{KriRin10}, we give a precise characterization of the
weak*-lower semicontinuous envelope $\overline{\mathfrak{F}}$
on~$\BVA(\Omega)$, see Proposition~\ref{prop:Fbar-extension}.

In consequence, a merger with Theorem~\ref{thm:main1} allows us to
formulate the minimisation problem with Dirichlet data~$u_0$ purely in
terms of~$\BVA(\Omega)$, see Corollary~\ref{cor:Fu0-lsc}. We demonstrate both the
existence of minima and the absence of a Lavrentiev-gap with respect
to the Dirichlet class $u_0+\WA_0(\Omega)$, see Thm. \ref{thm:main2}.
\subsection{Organisation of the paper}
The paper is organised as follows. In
Section~\ref{sec:funct-bound-opa} we fix notation, introduce the assumptions on the
differential operators~$\opA$ and collect elementary implications for the Sobolev--type 
spaces~$\WA(\Omega)$ and the spaces of functions of bounded $\A$--variation~$\BVA(\Omega)$. In
Section~\ref{sec:projpoinc} we introduce local projection operators onto
the nullspace~$N(\opA)$ on balls and derive \Poincare{}--type inequalities. In
Section~\ref{sec:traces}, we construct the trace
operator~$\trace\,:\, \BVA(\Omega) \to L^1(\partial \Omega;\mathcal{H}^{n-1})$ and thereby give the proof of Theorem~\ref{thm:main1}. Moreover, we establish a Gau{\ss}--Green formula and a gluing lemma for $\bvA$--maps. The final Section~\ref{sec:variational-problems} is dedicated to the existence
of~$\BVA$--minimisers of $\opA$--quasiconvex variational problems with linear growth subject to given Dirichlet boundary data.
{\small
\subsection*{Acknowledgments}

The authors wish to thank Jan Kristensen for numerous helpful
discussions, comments and reading a preliminary version of this paper. The third author further acknowledges
financial support and hospitality of the Max--Planck--Institut f\"{u}r
Mathematik in den Naturwissenschaften during a research stay in
Leipzig in May 2017, where parts of this project were concluded.}

\newpage

\section{Functions of Bounded $\opA$-Variation}
\label{sec:funct-bound-opa}

In this section we introduce spaces of functions of bounded
variation associated with a differential operator~$\opA$.
\subsection{General Notation}
To avoid too many different constants throughout, we write
$a \lesssim b$ if there exists a constant~$c$ (which does not depend
on the crucial quantities) with $a \leq c\, b$. If $a \lesssim b$ and
$b \lesssim a$, we also write~$a \eqsim b$. By $\ell(B)$ we denote the
diameter of a ball~$B$ and by~$\abs{B}$ its $n$--dimensional Lebesgue measure. We write
$d(\cdot,\cdot)$ for the usual euclidean distance. For the euclidean inner product of $a,b\in\R^{m}$ we use the equivalent notations $\langle a,b\rangle$ or $a\cdot b$. Given $f\in\lebe_{\locc}^{1}(\R^{n};\R^{K})$ and a measurable subset $U\subset\R^{n}$ with $|U|>0$, we use the equivalent notations
\begin{align*}
\dashint_{U} f(x) \,dx := \langle f\rangle_{U}:=\abs{U}^{-1} \int_U f(x)\,dx
\end{align*}
for the mean value integral. Lastly, for notational simplicity, we shall often surpress the possibly vectorial target space when dealing with function spaces and, e.g., write $\lebe^{1}(\R^{n})$ instead of $\lebe^{1}(
\R^{n};\R^{N})$, but this will be clear from the context. 
\subsection{Function Space Setup}
\label{sec:funct-analyt-setup}
Let $\A$ be given by \eqref{eq:form}. The corresponding \emph{dual (or formally adjoint) operator} $\opA^*$ is the differential operator on $\R^{n}$ from $\RK$ to $\RN$ given
by
\begin{align}
  \opA^*:=\sum_{\alpha=1}^n \A_\alpha^{*}\partial_\alpha,
\end{align}
where each $\A_\alpha^{*}$ is the adjoint matrix of $\A_\alpha$. For an open domain $\Omega\subset\R^{n}$ we define the \emph{Sobolev
  space $\WA(\Omega)$ associated to the operator~$\opA$} by
\begin{align}
  \label{eq:def-W1A}
  \sobo^{\A,1}(\Omega)  =  \sobo^{\A,1}(\Omega;\RN)
  &:=\big\{u\in\lebe^{1}(\Omega;\RN)\colon\;\A 
    u\in\lebe^{1}(\Omega;\RK) \big\}.
\end{align}
This is a Banach space with respect to the norm
\begin{align}
  \label{eq:def-W1A-norm}
  \norm{u}_{\WA(\Omega)} &:= \norm{u}_{L^1(\Omega)} + \norm{\opA u}_{L^1(\Omega)}.
\end{align}
We moreover define the \emph{total $\A$--variation} of
$u\in\lebe_{\locc}^{1}(\Omega;\RN)$ by
\begin{align}
  \label{eq:deftotalAvariation}
  |\A u|(\Omega):=\sup\Bigg\{\int_{\Omega}\langle
  u,\opA^*\varphi\rangle\dif
  x\colon\;\varphi\in\hold_{c}^{1}(\Omega;\RK),\;|\varphi|\leq
  1\Bigg\} 
\end{align}
and consequently say that $u$ is of \emph{bounded $\A$--variation} if
and only if $u\in\lebe^{1}(\Omega;\RN)$ and $|\A
u|(\Omega)<\infty$.
Denoting $\mathcal{M}(\Omega;\RK)$ the finite $\RK$--valued Radon measures on
$\Omega$, by the Riesz
representation theorem this amounts to
\begin{align} \label{eq:def-BVA}
  \bvA(\Omega):=\big\{u\in\lebe^{1}(\Omega;\R^{N})\colon\;\A
  u\in\mathcal{M}(\Omega;\RK) \big\}.  
\end{align}
Here, the shorthands $\A u\in\lebe^{1}$ or $\A u\in\mathcal{M}$ above have to be
understood in the sense that the distributional differential
expressions $\A u$ can be represented by $\lebe^{1}$--functions or
Radon measures, respectively. The norm
\begin{align}
  \label{eq:def-BVA-norm}
  \norm{u}_{\BVA(\Omega)} &:= \norm{u}_{L^1(\Omega)} + \abs{\opA u}(\Omega)
\end{align}
makes~$\BVA(\Omega)$ a Banach space. However, due to the lack of good
compactness properties, the norm topology turns out not useful in many
applications and one needs to consider weaker topologies. We now
introduce the canonical generalisations of well--known convergences in
the full-- or symmetric gradient cases, see~\cite{AmbFusPal00}.  Let
$u\in\BVA(\Omega)$ and $(u_k)\subset\BVA(\Omega)$. We say that
\begin{itemize}
\item $(u_k)$ converges to $u$ in the \emph{weak*--sense} (in symbols
  $u_k\wstar u$) if and only if $u_k\to u$ strongly in
  $\lebe^{1}(\Omega;\R^{N})$ and $\A u_k\wstar \A u$ in the weak*--sense of
  $\R^{K}$--valued Radon measures on $\Omega$ as $k\to\infty$.
\item $(u_k)$ converges to $u$ in the \emph{strict sense} (in symbols
  $u_k\strict u$) if and only if $d_{s}(u_k,u)\to 0$ as $k\to\infty$,
  where for $v,w\in\BVA(\Omega)$ we set
  \begin{align*}
    d_{s}(v,w):=\int_{\Omega}|v-w|\dif x + \bigabs{|\A v|(\Omega)-|\A
    w|(\Omega)}. 
  \end{align*}
\item $(u_k)$ converges to $u$ in the \emph{area-strict sense} (in
  symbols $u_k\areato u$) if
  and only if
  \begin{align*}
    \int_{\Omega}\sqrt{1+\left\vert\tfrac{\dif\A
    u_k}{\dif\mathscr{L}^{n}}\right\vert^{2}}\dif\mathscr{L}^{n}
    +|\A^{s}u_k|(\Omega) 
    \to \int_{\Omega}\sqrt{1+\left\vert\tfrac{\dif\A
    u}{\dif\mathscr{L}^{n}}\right\vert^{2}}\dif\mathscr{L}^{n}+|\A^{s}u|
    (\Omega),\qquad k\to\infty, 
  \end{align*}
  where
  $\A v = \tfrac{\dif\A
    v}{\dif\mathscr{L}^{n}}\mathscr{L}^{n}+\tfrac{\dif\A v}{\dif
    |\A^{s}v|}|\A^{s}v|$
  is the Radon--Nikod\v{y}m decomposition of
  $\A v\in\mathcal{M}(\Omega;\RK)$ with respect to the Lebesgue measure $\mathscr{L}^{n}$.
\end{itemize}
Strictly speaking, these notions are reserved for the $\bv$--versions and hence the above notions have to be read as $\A$--weak*, $\A$--strict and $\A$--area strict convergence. However, to keep terminology simple, we tacitly assume that the differential operator $\A$ is fixed throughout and stick to the above terminology.

Note that the $\opA$-variation is sequentially lower semicontinuous
with respect convergence in the weak*--sense, i.e., if
$u_k \wstar u$, then
$|\A u|(\Omega)\leq\liminf_{k\to\infty}|\A u_k|(\Omega)$.  Moreover,
if $u_k \in\BVA(\Omega)$ is a bounded sequence with $u_k \weakto u$ in
$L^1(\Omega;\R^{N})$, then already~$u_k \weakastto u$. Finally, if $\Omega$ is open and bounded with Lipschitz boundary, then it is easy to conclude by the theorem of Banach--Alaoglu that if $(u_{k})\subset\bvA(\Omega)$ is uniformly bounded in the $\bvA$--norm, then there exists $u\in\bvA(\Omega)$ and a subsequence $(u_{k(j)})$ of $(u_{k})$ such that $u_{k(j)}\wstar u$ as $j\to\infty$ in the sense specified above. We shall often refer to this as the \emph{weak*--compactness principle (for $\bvA$)}.

\subsection{Assumptions on the Differential Operator~$\opA$}
\label{sec:assumpt-diff-oper}


For our trace result we need some structure on~$\opA$ which we
introduce now.

Let $\A$ be given by \eqref{eq:form}. Then~$\opA$ induces a bilinear
pairing $\otimes_{\A}\colon\RN\times \Rn\to \RK$ by
\begin{align}
  \label{eq:bilinear}
  v\otimes_{\A} z:=\sum_{\alpha=1}^n z_\alpha\A_\alpha v,\qquad
  \text{for $z \in \Rn$ and $v \in \RN$}.  
\end{align} 
For all $\varphi\in\hold^{1}(\R^{n})$ and
$v\in\hold^{1}(\R^{n};\RN)$ we have
\begin{align}
  \label{eq:productrule}
  \A(\varphi v)=\varphi\,\A v+v \otimes_{\A}\nabla\varphi. 
\end{align}
Note that if~$\opA$ is the usual gradient, then~$\otimes_\opA$ can be
identified with the usual dyadic product $\otimes$, and if $\opA$ is the symmetric gradient, then $\otimes_{\opA}$ is given by the symmetric tensor product $\odot$. 

Recalling the notions of $\R$-- and $\mathbb{C}$--ellipticity from Section~\ref{sec:mainresults}, we now pass on to a more detailled discussion and begin with linking them to the \emph{type--(C)} condition as introduced in \cite{Kal94}. 
\begin{remark}
  \label{rem:Kal94}
  The operator~$\opA$ is $\setC$-elliptic if and only if it is of
  type~$(C)$ in the sense of~\cite{Kal94}. More precisely, since $\opA_\alpha[\xi]$ is a linear operator from $\RN$ to $\RK$ for each $\xi\in\R^{n}$,
  we find coefficients $\opA_{\alpha,j,k}$ such that
  \begin{align*}
    \big(\opA[\xi]\eta\big)_k  &=: \sum_{\alpha=1}^n \sum_{j=1}^N
                                 \opA_{\alpha,j,k} \xi_\alpha  \eta_j.
  \end{align*}
  for every for $\xi \in \Rn$ and $\eta \in \RN$. Then
  \begin{align*}
    \mathbb{P}_{j,k} u := \sum_{\alpha=1}^n
    \opA_{\alpha,j,k} \partial_\alpha u_j
  \end{align*}
  for $k=1,\dots,K$ is the family of scalar differential operators as used
  in~\cite{Kal94}. The corresponding symbols are
  \begin{align*}
    \mathbb{P}_{j,k}(\xi) := \sum_{\alpha=1}^n
    \opA_{\alpha,j,k} \xi_\alpha
  \end{align*}
  with $j=1,\dots, N$ and $k=1,\dots,K$.
  Now according to~\cite{Kal94} the family~$(\mathbb{P}_k)_k$ is of
  type~$(C)$ if and only if $(\mathbb{P}_{j,k}(\xi))_{j,k}$ has rank~$K$
  for all $\eta \in \setC^n \setminus \set{0}$. 
  Since
  \begin{align*}
    \sum_{j=1}^N \sum_{k=1}^K \mathbb{P}_{j,k}(\xi) \eta_j = \sum_{\alpha=1}^n
    \sum_{j=1}^N \sum_{k=1}^K  \opA_{\alpha,j,k} \xi_\alpha \eta_j = \opA[\xi]\eta
  \end{align*}
  this is equivalent to the injectivity of~$\opA[\xi]$ for all
  $\eta \in \setC^N \setminus \set{0}$, which is exactly the
  $\setC$-ellipticity of~$\opA$.
\end{remark}
We now turn to some examples which shall refer to frequently.
\begin{example}\label{ex:opexamp}
In what follows, we carefully examine the gradient, symmetric and trace--free symmetric gradient operators. As these typically map $\R^{N}$ to the matrices $\R^{N\times n}$ instead of a vector in $\RK$, we henceforth put $K=Nn$ and identify
  $\RK$ with $\R^{N \times n}$.
  \begin{enumerate}
  \item \label{itm:opexamp1} Let $\opA u := \nabla u$. Then $N(\opA)$
    just consists of the constants and
    \begin{align*}
      (v \otimes_\nabla z)_{j,k} &= v_j z_k.
    \end{align*}
    $\opA$ has a finite dimensional nullspace and is~$\setC$-elliptic,
    since
    \begin{align*}
      \abs{\A[\xi]\eta}^2 &= \abs{\xi}^2 \abs{\eta}^2.
    \end{align*}
  \item \label{itm:opexamp2} Let 
    $\A u:=\mathcal{E}(u) := \frac 12 (\nabla u + (\nabla u)^T)$ with
    $N=n$. Then $N(\mathcal{E})$ just consists of the generators of rigid
    motions, i.e.,
    \begin{align*}
      N(\mathcal{E}) 
      &= \set{ x \mapsto Ax + b\,:\, \text{$A \in \R^{n \times
        n}$, $A=-A^T$, $b \in \Rn$}} 
    \end{align*}
    and 
    \begin{align*}
      (v \otimes_{\mathcal{E}} z)_{j,k} &= \tfrac 12 (v_j z_k + v_k z_j).
    \end{align*}
    $\mathcal{E}$ has a finite dimensional nullspace and
    is~$\setC$-elliptic, since
    \begin{align*}
      \abs{\A[\xi]\eta}^2 &= \tfrac 12\abs{\xi}^2 \abs{\eta}^2 + \tfrac
                           12 \abs{\skp{\xi}{\eta}}^2.
    \end{align*}
  \item \label{itm:opexamp3} Let 
    $\A u:=\mathcal{E}^D(u) = \frac 12 (\nabla u + (\nabla u)^T)- \frac 1n
    \divergence(u) \identity_n$ with $N=n$.  Then
    \begin{align*}
      (v \otimes_{\mathcal{E}^D} z)_{j,k} &= \tfrac 12 (v_j z_k + v_k z_j) -
                                 \frac{1}{n} \delta_{j,k} \sum_{l=1}^n v_l z_l
    \end{align*}
    and
    \begin{align*}
      \abs{\A[\xi]\eta}^2 &= \tfrac 12\abs{\xi}^2 \abs{\eta}^2 + \tfrac
                           12 \abs{\skp{\xi}{\eta}}^2 - \tfrac{1}{n}
                           \skp{\xi}{\bar{\eta}}^2. 
    \end{align*}
    If $n \geq 3$, then $\opA$ is $\setC$-elliptic and it has the
    finite dimensional nullspace
    \begin{align*}
      \qquad\quad  N(\mathcal{E}^D) 
      &= \bigset{ x \mapsto Ax + b+ (2(a \cdot x) x- \abs{x}^2 a)
        \,:\, \text{$A\in \setR^{n\times
        n}$, $A=-A^T$, $a,b \in \Rn$}}. 
    \end{align*}
    Elements of~$N(\mathcal{E}^D)$ are also known as \emph{conformal
      killing vectors} \cite{Dain06}.

    If $n=2$, then $\opA$ is only $\setR$-elliptic, but
    not~$\setC$-elliptic. Indeed, $\opA[\xi]\eta=0$ for $\xi=(1,i)^T$
    and $\eta=(1,-i)^T$. Moreover, the nullspace $N(\opA)$ is of
    infinite dimension:  Indeed, if we identify $\setR^2 \cong \setC$,
    then the kernel of~$\mathcal{E}^D$ consists of the holomorphic
    functions. We will substantially use this property in the proofs of
    Lemma~\ref{lem:FDNchar} and Theorem~\ref{thm:no-trace}.
  \end{enumerate}
\end{example}
We now draw some consequences of the single ellipticity conditions and link them to the finite dimensionality of the nullspace of $\A$.
\begin{lemma}
  \label{lem:upperlower}
  Let $\opA$ be $\setK$-elliptic with $\setK=\setR$ or
  $\setK=\setC$. Then there exists two constants
  $0<\kappa_1\leq \kappa_2<\infty$ such that
  \begin{align*}
    \kappa_1|v|\,|z|\leq |v\otimes_{\A}z|\leq \kappa_2|v|\,|z| \qquad
    \text{for all $v\in\setK^N$ and  $z\in\setK^n$
    }.
  \end{align*} 
\end{lemma}
\begin{proof}
  By scaling it suffices to assume $\abs{v} = \abs{z}=1$. We
  have~$\abs{v \otimes_\opA z}>0$, since $\opA$ is $\setK$-elliptic. Now the
  claim follows by compactness of $\set{(v,z)\,:\, \abs{v}=\abs{z}=1}$
  and continuity.
\end{proof}

\begin{lemma}
  \label{lem:FDN-R-elliptic}
  Let $\opA$ have a finite dimensional nullspace. Then~$\opA$ is $\setR$-elliptic.
\end{lemma}
\begin{proof}
  We proceed by contradiction. Assume that $\opA$ is
  not $\setR$-elliptic. Then there exists
  $\xi \in \Rn \setminus \set{0}$ and $\eta \in \RN \setminus \set{0}$
  with $\opA[\xi]\eta=0$.  For every $f \in \hold_{c}^{1}(\setR;\setR)$ we
  define~$u_f(x) := f(\skp{\xi}{x}) \eta$. Then
  $(\opA u_f)(x) = \opA[\xi] \eta\, f(\skp{\xi}{x}) =
  0$. Since~$\eta\neq 0$ and~$\xi\neq 0$, the mapping~$f \mapsto u_f$
  is injective. Therefore, the set~$\set{u_f\,:\, f \in \hold_{c}^{1}(\setR)}$
  is an infinite dimensional subspace of~$N(\opA)$. This contradicts the fact 
  that $\opA$ has finite dimensional nullspace.
\end{proof}
\begin{lemma}
  \label{lem:FDNchar}
  Let $\opA$ have a finite dimensional nullspace. Then $\opA$ is $\setC$-elliptic.
\end{lemma}
\begin{proof}
  Since~$\opA$ has finite dimensional nullspace, it is
  $\setR$-elliptic by Lemma~\ref{lem:FDN-R-elliptic}.

  We proceed by contradiction, so assume that~$\opA$ is not
  $\setC$-elliptic. Then there
  exists~$\xi \in \setC^n \setminus \set{0}$ and
  $\eta \in \setC^N \setminus \set{0}$ with
  $0 = \opA[\xi]\eta = \eta \otimes_{\opA} \xi$. We split~$\xi$
  and~$\eta$ into their real and imaginary parts by
  $\xi=: \xi_1 + i \xi_2$ and $\eta =: \eta_1 + i \eta_2$. Then
  $\opA[\xi]\eta=0$ implies
  \begin{alignat}{2}\label{eq:CR}
    \mathbb{A}[\xi_{1}]\eta_{1}-\mathbb{A}[\xi_{2}]\eta_{2} &=0
    \qquad \text{and} \qquad &
    \mathbb{A}[\xi_{1}]\eta_{2}+\mathbb{A}[\xi_{2}]\eta_{1} &=0.
  \end{alignat}
  We will show not that $\xi_1$ and $x_2$, resp. $\eta_1$ and
  $\eta_2$, are linearly independent.

  We begin with the linear independence of~$\xi_1$ and $\xi_2$.  If
  $\xi_1=0$, then $\xi_2 \neq 0$ and then the $\setR$-ellipticity
  of~$\opA$ and~\eqref{eq:CR} implies~$\eta_1=\eta_2=0$, which
  contradicts~$\eta\neq 0$. By the same argument, also~$\xi_2=0$ is
  not possible. Hence, we have $\xi_1 \neq 0$ and $\xi_2 \neq 0$. We
  now show the linear independence of~$\xi_1$ and $\xi_2$ by
  contradiction, so let us assume that~$\xi_2 = \lambda \xi_1$ with
  $\lambda \neq 0$. Then it follows from~\eqref{eq:CR} that
  \begin{align*}
    \opA[\xi_1]\eta_1 
    &= 
      \opA[\xi_2]\eta_2 = \lambda \opA[\xi_1] \eta_2 = -\lambda
      \opA[\xi_2]\eta_1 = -\lambda^2 \opA[\xi_1][\eta_1].
  \end{align*}
  This implies $\opA[\xi_1][\eta_1]=0$. Hence by $\setR$-ellipticity
  of~$\opA$ and $\xi_1\neq 0$, we get $\eta_1=0$. Now,~\eqref{eq:CR}
  implies $\opA[\xi_2][\eta_2]=0$, so again the $\setR$-ellipticity
  of~$\opA$ gives~$\eta_2=0$. Overall, $\eta=0$, which is a
  contradiction. This proves that~$\xi_1$ and $\xi_2$ are linearly
  independent.

  The proof of the linear independence of~$\eta_1$ and $\eta_2$ is
  completely analogous. Indeed, $\eta_1 = \gamma \eta_2$ implies
  $\opA[\xi_1]\eta_1 = - \gamma^2 \opA[\xi_1]\eta_1$, so
  $\opA[\xi_1][\eta_1]=0$. As above this implies~$\eta=0$, which is a
  contradiction.

  Let us define now $\tau\,:\, \Rn \to \setC$ and
  $\sigma\,:\, \setC \to \RN$ by
  \begin{align*}
    \tau(x) &:= \skp{\xi}{x}
              = \skp{\xi_1}{x} + i \skp{\xi_2}{x},
              \\
    \sigma(z) &:= \Re(z) \eta_1 - \Im(z) \eta_2.
  \end{align*}
  Let $\mathcal{O}(\setC)$ denote the set of holomorphic
  functions on $\setC$. Then~$\dim (\mathcal{O}(\setC))= \infty$.  Moreover,
  for~$f \in \mathcal{O}(\setC)$ we have $\partial_{\bar{z}} f(z) = 0$
  in the sense of complex derivatives. Let us define
  $h_f \colon \Rn\to \RN$ by $h_f:= \sigma \circ f \circ \tau$. Our
  goal is to prove $\opA h_f =0$. We identify in the
  following~$\setC$ with $\setR^2$.  With the chain rule we conclude
  \begin{align}
    \label{eq:FDNchar1}
    \begin{aligned}
      (\opA h_f)(x) &= \quad \opA[\xi_1] \eta_1 (\partial_1
      f_1)(\tau(x)) - \opA[\xi_1] \eta_2 (\partial_1 f_2)(\tau(x))
      \\
      &\quad\, + \opA[\xi_2] \eta_1 (\partial_2 f_1)(\tau(x)) -
      \opA[\xi_2] \eta_2 (\partial_2 f_2)(\tau(x)) .
    \end{aligned}
  \end{align}
  Using the Cauchy-Riemann equations 
  $\partial_1 f_1 = \partial_2 f_2$ and
  $\partial_1 f_2 = -\partial_2 f_1$ and~\eqref{eq:CR} we get
  \begin{align*}
    (\opA h_f)(x) 
    &= (\opA[\xi_1] \eta_1 -\opA[\xi_2] \eta_2) (\partial_1 f_1)(\tau(x))
    + (\opA[\xi_1] \eta_2+\opA[\xi_2] \eta_1) (\partial_2
      f_1)(\tau(x)) =0.
  \end{align*}
  So for each~$f \in \mathcal{O}(\setC)$, we constructed an
  $h_f \,:\, \Rn \to \RN$ such that $\opA h_f =0$. We need to show
  that $\dim(\set{h_f\,:\, f\in \mathcal{O}(\setC)})=\infty$. For
  this, it suffices to show that the linear mapping $f \mapsto h_f$ is
  injective. Recall that $h_f = \sigma \circ f \circ \tau$. Hence, it
  suffices to show that $\sigma$ is injective and that~$\tau$ is
  surjective. This, however, follows from the fact that $\xi_1$ and
  $\xi_2$, resp. $\eta_1$ and $\eta_2$, are linearly independent.
  This concludes the proof.
\end{proof}
\begin{theorem}
  \label{thm:opA-equiv}
  The following are equivalent.
  \begin{enumerate}
  \item \label{itm:opA-equiv1} $\opA$ has a finite dimensional nullspace.
  \item \label{itm:opA-equiv2} $\opA$ is $\setC$-elliptic.
  \item \label{itm:opA-equiv3} There exists~$l \in \setN$ with $N(\opA) \subset
    \mathscr{P}_l$, where $\mathscr{P}_l$ denotes the set of
    polynomials with degree less or equal to~$l$.
  \end{enumerate}
\end{theorem}
\begin{proof}
  Lemma~\ref{lem:FDNchar} proves
  \ref{itm:opA-equiv1}$\Rightarrow$\ref{itm:opA-equiv2}. Obviously,
  \ref{itm:opA-equiv3}$\Rightarrow$\ref{itm:opA-equiv1}. It remains to
  show \ref{itm:opA-equiv2}$\Rightarrow$\ref{itm:opA-equiv3}.

  Since~$\opA$ is $\setC$-elliptic, it is of type--(C) in the sense of
  \cite{Kal94}, see Remark~\ref{rem:Kal94}.
  Fix~$\omega \in \hold_{c}^{\infty}(B(0,1))$ with
  $\int_{B(0,1)} \omega\,dx=1$. Then for an arbitrary ball $B$, we obtain by dilation and translation a
  function~$\omega_B \in \hold_{c}^{\infty}(B)$ with
  $\int_B \omega_B(y)\,dy = 1$.  For every~$l \in \setN_0$ let
  $\mathcal{P}_{B}^l$ denote the \emph{averaged Taylor polynomial} with
  respect to~$B$ of order~$l$ (see~\cite{dup_scott78}), i.e.
  \begin{align*}
    \mathcal{P}^l_B u(x) &:= \int_B \sum_{\abs{\beta} \leq l} \partial^\beta_y
                         \bigg( \frac{(y-x)^\beta}{\beta!} \omega_B(y)
                         \bigg) u(y)\,dy.
  \end{align*}
  The formula is obtained by multiplying Taylor's polynomial of
  order~$l$ by the weight~$\omega_B$ and integrating by parts. Note that
  $\mathcal{P}_{B}^l u \in \mathscr{P}_l$.

  It follows from the representation formula of \cite{Kal94}, Theorem
  4], that for all~$x\in B$
  \begin{align}
    \label{eq:Plu-riesz}
    \abs{u(x) - (\mathcal{P}^l_B u)(x)} &\leq c\, \int_B \frac{\abs{(\opA
                                    u)(y)}}{\abs{x-y}^{n-1}}\,dy,
  \end{align}
  for some $l \in \setN_0$ (which is fixed from now on) and
  all~$u\in C^\infty(B)$. We do not know the exact value of~$l$, but
  at least~$l$ is so large that $N(\opA) \subset \mathscr{P}_l$ (there is,
  however, an upper bound for~$l$ in terms of~$n$ and~$N$.)

  Now, let $v \in N(\opA)$, i.e. $v\in\mathcal{D}'(\Rn;\RN)$ with
  $\opA v = 0$ in the distributional sense. Let $\phi_\epsilon$ denote
  a standard mollifier, i.e., $\varphi_{\epsilon}(x):=\epsilon^{-n}\varphi(x/\varepsilon)$ with a radially symmetric function $\varphi\in\hold_{c}^{\infty}(\ball;[0,1])$ with $\int_{\ball}\varphi\dif x =1$. Then $v * \phi_\epsilon \in C^\infty(\Rn)$ and
  $\opA (v * \phi_\epsilon) = (\opA v) * \phi_\epsilon =0$.  Hence,
  it follows from~\eqref{eq:Plu-riesz} that $v * \phi_\epsilon \in
  \mathscr{P}_l(\Rn)$. This implies~$v \in \mathscr{P}_l(\Rn)$ as
  desired. The proof is complete.
\end{proof}
\begin{remark}
  \label{rem:VS}
  Let us compare our conditions with the ones of Van
  Schaftingen~\cite{Van13}, building on the fundamental work of
  Bourgain \& Brezis \cite{BoBr1,BoBr2}. According to~\cite{Van13} the operator~$\opA$ is
  \emph{cancelling}\footnote{The definition of \emph{cancelling}
    in~\cite{Van13} is given in terms of the annihilating
    operator~$\mathbb{L}$ from the exact sequence
    in~\eqref{eq:complex}. However, it translates in our setting
    to~\eqref{eq:cancelling}.}  if
  \begin{align}
    \label{eq:cancelling}
    \bigcap_{\xi \neq 0} \opA[\xi](\RN) &= \set{0}.
  \end{align}
  It has been shown in~\cite[Theorem~1.4]{Van13} that whenever $\opA$
  is $\setR$-elliptic and cancelling, then we have the Sobolev--type
  inequality
  \begin{align}\label{eq:VSineq2}
    \|u\|_{\lebe^{\frac{n}{n-1}}(\R^{n};\RN)}\leq C\|\A u\|_{\lebe^{1}(\R^{n};\RK)}
  \end{align}
  for all $u\in\hold_{c}^{\infty}(\R^{n};\R^{N})$. Moreover, the
  $\setR$-ellipticity and cancellation property of~$\opA$ is necessary
  for such inequality.

  For our result on traces we need $\setC$-ellipticity of~$\opA$. So
  the natural question arises how $\setC$-ellipticity  compares to the
  canceling property. It will been shown in~\cite{GmeRai17}
  that $\setC$-ellipticity implies the canceling
  property but not vice-versa. Indeed, the operator
  \begin{align*}
    \opA(u) 
    &:= 
      \begin{pmatrix}
        \frac 12 \partial_1 u_1 - \frac 12 \partial_2 u_2 & 
        \frac 12 \partial_1 u_2 + \frac 12 \partial_2 u_1 &
        \partial_3 u_1
        \\[1mm]
        \frac 12 \partial_1 u_2 + \frac 12 \partial_2 u_1 &
        \frac 12 \partial_1 u_1 - \frac 12 \partial_2 u_2 & 
        \partial_3 u_2
      \end{pmatrix}
  \end{align*}
  is $\setR$-elliptic and cancelling but it is not $\setC$-elliptic,
  since it fails the finite dimensional nullspace property (recall Thm. \ref{thm:opA-equiv}).

\end{remark}

\subsection{Smooth approximations in the interior}
\label{sec:smooth-appr}

In this section we show that functions from~$\WA(\Omega)$ and
$\BVA(\Omega)$ can be approximated in a certain sense by functions
from~$\sobo^{\A,1}(\Omega)\cap C^\infty(\Omega;\R^{N})$. The proof is in the spirit of \cite[Chpt.~5.2]{EvGa} and is included for the reader's convenience.

\begin{theorem}[Smooth Approximation]\label{thm:smoothapprox}
  Let $\Omega\subset\R^{n}$ be open. Then the following hold:
  \begin{enumerate}
  \item\label{item:smoothapprox1} The space
    $(\hold^{\infty}\cap \sobo^{\A,1})(\Omega)$ is dense in
    $\sobo^{\A,1}(\Omega)$ with respect to the norm topology.
  \item\label{item:smoothapprox2} The space
    $(\hold^{\infty}\cap \BVA)(\Omega)$ is dense in $\BVA(\Omega)$
    with respect to the area-strict topology.
  \end{enumerate}
\end{theorem}
\begin{proof}
 Fix $u \in \bvA(\Omega)$. For $k =2,3,\dots$ define
  $\Omega_k := \set{x \in \Omega\,:\, \frac 1{k+1} < d(x,\partial
    \Omega) < \frac 1{k-1}}$.
  Now pick a sequence $(\psi_k)$ such that for each $k\in\setN$,
  $\psi_k\in \hold^{\infty}_c(\Omega_k;[0,1])$ together with
  $\sum_k\psi_k=1$ globally in $\Omega$. Now let
  $\eta_\epsilon \colon\R^{n}\to\R$ be a standard mollifier (even and
  non-negative).

  For $j \in \setN$ and $k \in \setN$ we can find $\epsilon_{j,k}>0$
  such that
  \begin{enumerate}[label={(\roman{*})}]
  \item \label{itm:smooth1}
    $\spt(\eta_{\varepsilon_{j,k}}*(\psi_ku))\subset\Omega_k$,
  \item \label{itm:smooth2}
    $\|\psi_ku -
    \eta_{\varepsilon_{j,k}}*(\psi_ku)\|_{\lebe^{1}(\Omega)}<2^{-k-j}$,
  \item \label{itm:smooth3} 
    $\| u \otimes_\opA \nabla\psi_k - 
    \eta_{\varepsilon_{j,k}}*(u \otimes_\opA \nabla\psi_k)
    \|_{\lebe^{1}(\Omega)}<2^{-k-j}$.
  \item \label{itm:smooth4} If $u \in \WA(\Omega)$, we additionally
    require
    $\|\psi_k \opA u - \eta_{\varepsilon_{j,k}}*(\psi_k \opA
    u)\|_{\lebe^{1}(\Omega)}<2^{-k-j}$.
  \end{enumerate}
  This allows us to define $u_j \in C^\infty(\Omega)$ by
  $u_j:=\sum_{k\in\mathbb{N}}\eta_{\varepsilon_{j,k}}*(\psi_ku)$,
  which is well defined in~$L^1_{\loc}(\Omega)$, since the sum is locally
  finite. Then in $L^1_{\loc}(\Omega)$
  \begin{align*}
    u - u_j &= \sum_k \big(\psi_k u - \eta_{\epsilon_{j,k}} * (\psi_k u)\big).
  \end{align*}
  This and~\ref{itm:smooth2} implies $\norm{u-u_j}_{L^1(\Rn)} \lesssim 2^{-j}$.
  If $u \in \WA(\Omega)$, then~\ref{itm:smooth3} and
  ~\ref{itm:smooth4} imply $\norm{\opA u-\opA u_j}_{L^1(\Rn)}
  \lesssim 2^{-j}$. This proves~\ref{item:smoothapprox1}.

  It remains to prove $u_j \areato u$ for $j \to \infty$ for
  $u \in \BVA(\Omega)$. In fact, the proof is like in the
  standard~$\setBV$ case. For simplicity of notation we just show
  $u_j \strictto u$ for $j \to \infty$. The necessary changes to pass
  from strict convergence to area-strict convergence are just like
  in~\cite[Lemma B.2]{Bil03}.

  Since $u_j \to u$ in $L^1(\Rn)$ it follows by
  by the lower semicontinuity of the total $\opA$--variation that
  $|\opA u|(\Omega)\leq \liminf_{j \to \infty} |\opA u_j|(\Omega)$. It
  remains to prove
  $\limsup_{j \to \infty} \abs{\opA u_j}(\Omega) \leq \abs{\opA
    u}(\Omega)$.
  For this we invoke the dual characterisation
  \eqref{eq:deftotalAvariation} of the total $\A$--variation. Let
  $\varphi\in \hold^{1}_c(\Omega; \RK)$ with $|\varphi|\leq 1$ be arbitrary. We
  compute
  \begin{align*}
    \int_{\Omega}\langle u_j,\opA^*\varphi\rangle\dif x 
    & = \sum_{k}\int_{\Omega}\langle
      \eta_{\varepsilon_{j,k}}*(\psi_ku),\opA^*\varphi\rangle\dif x =
      \sum_{k}\int_{\Omega}\langle
      \psi_ku,\opA^*(\eta_{\varepsilon_{j,k}}*\varphi)\dif x 
    \\
    & = \sum_{k}\int_{\Omega}\langle
      u,\opA^*(\psi_k(\eta_{\varepsilon_{j,k}}*\varphi))\rangle\dif x
      -\sum_{k}\int_{\Omega}\langle
      u,(\eta_{\varepsilon_{j,k}}*\varphi) \otimes_{\opA^*} \nabla\psi_k\rangle\dif
      x
    \\
    &=: I_j +II_j.
  \end{align*}
  The sums are well defined, since $\phi \in \hold^{1}_c(\Omega)$ and
  $u_j = \sum_k \eta_{\epsilon_{j,k}} * (\psi_k u)$ in
  $L^1_{\loc}(\Omega)$. Now
  \begin{align*}
    \Bigabs{\sum_k \psi_k (\eta_{\epsilon_{j,k}} * \phi)}
    &\leq
      \sum_k \psi_k \abs{\eta_{\epsilon_{j,k}} * \phi} \leq
      \sum_k \psi_k \norm{\phi}_\infty = \norm{\phi}_\infty \leq 1.
  \end{align*}
  Therefore,
  \begin{align*}
    I_j &= \int_{\Omega}\Bigskp{
          u}{\opA^*\Big( \sum_k\psi_k(\eta_{\varepsilon_{j,k}}*\varphi) \Big)}\dif x
          \leq \abs{\opA u}(\Omega).
  \end{align*}
  Using $\sum_k\nabla\psi_k=0$ and $\phi \in \hold^{1}_c(\Omega)$, we
  now rewrite $II_j$ as
  \begin{align*}
    II_j
    &= \sum_{k}\int_{\Omega}\langle
      u,(\eta_{\varepsilon_{j,k}}*\varphi) \otimes_{\opA^*} \nabla\psi_k\rangle\dif
      x - \sum_{k}\int_{\Omega}\langle
      u,\varphi \otimes_{\opA^*} \nabla\psi_k\rangle\dif
      x
    \\
    & = \sum_{k} \int_\Omega \langle
      \eta_{\varepsilon_{j,k}}*(u \otimes_{\A} \nabla\psi_k)
      -(u \otimes_{\A} \nabla\psi_k), \varphi\rangle\dif
      x.
  \end{align*}
  Invoking~\ref{itm:smooth3} and $\norm{\phi}_\infty \leq 1$ we obtain
  $|II_j|\lesssim 2^{-j}$.  Hence, collecting estimates we obtain
  as desired
  $\limsup_{j \to \infty} \abs{\opA u_j}(\Omega) \leq \limsup_{j \to
    \infty} (\abs{\opA u}(\Omega)+ c\, 2^{-j}) =\abs{\opA u}(\Omega)$.
\end{proof}

\section{Projections and \Poincare{} Inequalities}
\label{sec:projpoinc}

In this section we derive several versions of \Poincare{}'s
inequality. We assume throughout the section that~$\opA$ is
$\setC$-elliptic (or, equivalently: $\opA$ has a finite dimensional nullspace, see Thm. \ref{thm:opA-equiv}).

\subsection{Projection Operator}
\label{sec:projection-operator}

We begin with some projection estimates.

For every ball~$B \subset \Rn$ and~$u \in L^2(B;\RN)$ we define $\Pi_B u$
as the $L^2$-projection of~$u$ onto $N(\opA)$. Hence,
\begin{align*}
  \int_B \abs{\Pi_B u}^2\,dx &\leq 
                               \int_B \abs{u}^2\,dx.
\end{align*}
Since $N(\opA)$ is finite dimensional, there exists a constant~$c > 0$
with
\begin{align}
  \label{eq:PiB-Linfty_L1}
  \norm{\Pi_B u}_{L^\infty(B)} &\le 
  c\, \dashint_B \abs{\Pi_B u}\,dx.
\end{align}
Indeed, this is clear for the unit ball and extends to general balls by dilation
and translation. It follows from this as usual that
\begin{align}
  \label{eq:PiB-L1}
  \dashint_B \abs{\Pi_B u}\,dx   &\leq c\,
                                   \dashint_B
                                   \abs{u}\,dx.
\end{align}
Thus, $\Pi_B$ can be extended to~$L^1(B;\RN)$ such that~\eqref{eq:PiB-L1}
remains valid.

\begin{lemma}
  \label{lem:inf-vs-mean}
Then there
  exists~$c\geq 1$ with
  \begin{align*}
    \inf_{q \in N(\opA)} \norm{u-q}_{L^1(B)} 
    &\le \norm{u - \Pi_B u}_{L^1(B)} \le c\,   \inf_{q \in N(\opA)} \norm{u-q}_{L^1(B)}.
  \end{align*}
\end{lemma}
\begin{proof}
  The first estimate is obvious. Now, for all~$q \in N(\opA)$ we have
  $\Pi_B q =q$. This and~\eqref{eq:PiB-L1} imply
  \begin{align*}
    \norm{u - \Pi_B
    u}_{L^1(B)} 
    &\leq
      \norm{u - q}_{L^1(B)} +
      \norm{\Pi_B (u-q)}_{L^1(B)} \leq c\,
      \norm{u - q}_{L^1(B)}.
  \end{align*}
  Taking the infimum over~$q \in N(\opA)$ proves the lemma.
\end{proof}

\subsection{\Poincare{} Inequalities}
\label{sec:poincare}

In this subsection we derive \Poincare{}--type
inequalities for $\WA$ and $\BVA$. Recall that for a ball $B$ we denote by~$\ell(B)$ its diameter.

\begin{theorem}
  \label{thm:poincare}
  There exists a constant~$c>0$ such that for all balls~$B$ and all $u
  \in \BVA(B)$ it holds
  \begin{align*}
    \inf_{q \in N(\opA)} \norm{u-q}_{L^1(B)}  \le \norm{u-\Pi_B
    u}_{L^1(B)} &\le c\, \ell(B)\,\abs{\opA u}(B),
  \end{align*}
  where $\Pi_B$ is the $L^2$-orthogonal projection onto~$N(\opA)$ from
  Subsection~\ref{sec:projection-operator}.
\end{theorem}
\begin{proof}
  By dilation and translation, it suffices to prove the claim for the
  unit ball $B=B(0,1)$.  
  Moreover, by smooth approximation (see Theorem~\ref{thm:smoothapprox})
  it suffices to consider $u \in C^\infty(B;\RN) \cap \WA(B)$.

  We use the averaged Taylor polynomals as in the proof of
  Theorem~\ref{thm:opA-equiv}. Recall that by~\eqref{eq:Plu-riesz} we
  have the estimate
  \begin{align}
    \label{eq:8}
    \abs{u(x) - (\mathcal{P}^lu)(x)} &\leq c\, \int_B \frac{\abs{(\opA
                                    u)(y)}}{\abs{x-y}^{n-1}}\,dy\qquad\text{for all}\;x\in B.
  \end{align}
  Since $\mathcal{P}^l u$ is not necessarily in the kernel
  of~$\opA$, we wish to replace it by $\Pi_B (\mathcal{P}^l)$. Thus,
  we start with
  \begin{align}
    \label{eq:6}
    \abs{u(x) - \Pi_B (\mathcal{P}^lu)(x)} 
    &\leq 
      \abs{u(x) - (\mathcal{P}^lu)(x)} +
      \abs{(\mathcal{P}^lu)(x) - (\Pi_B (\mathcal{P}^lu))(x)}.
  \end{align}
  Now, for any~$p \in \mathscr{P}_l$ there holds
  \begin{align}
    \label{eq:Pl-opA}
    \norm{p - \Pi_B p}_{L^\infty(B)} &\leq c\, \dashint_B \abs{\opA p}\,dx.
  \end{align}
  Indeed, both sides define a norm on the finite dimensional
  space~$\mathscr{P}_l / N(\opA)$ and vanish on~$N(\opA)$.
  Hence, for all~$x \in B$
  \begin{align}
    \label{eq:5}
    \begin{aligned}
      \abs{(\mathcal{P}^lu)(x) - (\Pi_B (\mathcal{P}^lu))(x)} &\le
      \norm{\mathcal{P}^lu - \Pi_B (\mathcal{P}^lu)}_{L^\infty(B)}
      \\
      &\le c\, \dashint_B \abs{\opA (\mathcal{P}^l u)}\,dx.
    \end{aligned}
  \end{align}
  The definition of the averaged Taylor polynomial implies that
  \begin{align}
    \label{eq:Pl-commutes}
    \opA (\mathcal{P}^l u) = \mathcal{P}^{l-1} (\opA u),
  \end{align}
  where $\mathcal{P}^{-1} u := 0$ if~$l=0$. The $L^1$-stability of the
  averaged Taylor polynomial gives
  \begin{align}
    \label{eq:Pl-stab}
    \norm{\mathcal{P}^{l-1} (\opA  u)}_{L^1(B)} &\le c\,
                                                  \norm{\opA u}_{L^1(B)}.
  \end{align}
  Now, \eqref{eq:Pl-opA} and~\eqref{eq:Pl-stab} yield
  \begin{align*}
    \abs{(\mathcal{P}^lu)(x) - (\Pi_B (\mathcal{P}^lu))(x)} 
    &\le c\, \ell(B)
      \dashint_B
      \abs{\opA u}\,dy
      \leq c\, \int_B \frac{\abs{(\opA
      u)(y)}}{\abs{x-y}^{n-1}}\,dy.
  \end{align*}
  So,~\eqref{eq:8} and~\eqref{eq:6} imply the estimate
  \begin{align}
    \label{eq:u-PiBPl-riesz}
    \abs{u(x) - (\Pi_B \mathcal{P}^lu)(x)} &\leq c\, \int_B \frac{\abs{(\opA
                                             u)(y)}}{\abs{x-y}^{n-1}}\,dy.
  \end{align}
  Now, integration over~$x \in B$ gives
  \begin{align*}
    \int_B \abs{u - \Pi_B(\mathcal{P}^l u)}\,dx 
    &\le c\, \int_B \int_B \frac{\abs{(\opA
      u)(y)}}{\abs{x-y}^{n-1}}\,dy\,dx
      \\
    &\le c\, \int_B \abs{(\opA u)(y)} \int_B \abs{x-y}^{1-n}\,dx \,dy
      \\
    &\le c\, \ell(B) \int_B \abs{\opA u} \,dy.
  \end{align*}
  We have shown
  \begin{align}
    \label{eq:pre-poincare}
    \norm{u  - \Pi_B(\mathcal{P}^l u)}_{L^1(B)} 
    &\le c\, \ell(B)
      \norm{\opA
      u}_{L^1(B)}. 
  \end{align}
  The rest follows by Lemma~\ref{lem:inf-vs-mean}.

\end{proof}
\begin{theorem}
  \label{thm:poincare2}
  Let $B'$ and $B$ are two balls with $B' \subset B$ and
  $\ell(B) \lesssim \ell(B')$. Then for all~$u \in \BVA(B)$ with $u=0$
  on~$B'$, there holds
  \begin{align*}
    \norm{u}_{L^1(B)} &\leq c\, \ell(B) \abs{\opA u}(B).
  \end{align*}
  The constant only depends on the ratio $\ell(B)/\ell(B')$.
\end{theorem}
\begin{proof}
  We use the same construction as in the proof of
  Theorem~\ref{thm:poincare}. However, we
  choose~$\omega \in \hold^\infty_c(B)$ in the construction of the
  averaged Taylor polynomial additionally as
  $\omega \in \hold^\infty_c(B')$. This implies that~$\mathcal{P}^lu$
  only depends on the values of~$u$ on~$B'$.  Hence, we obtain
  $\mathcal{P}^lu=0$. Hence, Theorem \ref{thm:poincare} proves the claim.
\end{proof}
Finally, let us remark that variants of Poincar\'{e}--type inequalities can also be established along the lines of \cite[Lem.~8.3.1]{AdHe} or \cite[Chpt.~4]{Zie89}. However, this requires additional extension and compactness arguments which need to be proven independently. 
\section{Traces}
\label{sec:traces}

In this section we show that the space of functions bounded $\A$--variation admits a continuous trace operator to~$L^1(\partial\Omega)$
if and only if $\opA$ is $\setC$-elliptic (or, equivalently: $\opA$ has a finite dimensional nullspace, see Thm. \ref{thm:opA-equiv}).

\subsection{Assumptions on the Domain}
\label{sec:assumptions-domain}

In order to ensure a proper trace we need to make certain regularity assumptions
on~$\Omega$. Our results include all Lipschitz graph
domains. However, we will consider even more general domains. Indeed,
the non-tangentially accessible domains (NTA domains) provide a
natural setting for our construction of the trace operator. We refer
to~\cite{HofMitTay10} for more information on NTA domains.

We begin with the necessary conditions on our domain.
\begin{definition}[Interior/Exterior Corkscrew Condition]
  Let $\Omega \subset \Rn$.
  \begin{enumerate}
  \item We say that~$\Omega$ satisfies the \emph{interior corkscrew
      condition} if there exist~$R>0$ and $M > 2$ such that for all $x
    \in \partial \Omega$ and all~$r \in (0,R)$ there exists a $y \in
    \Omega$ such that
    \begin{align*}
      \frac{1}{M} r\le \abs{x-y} &\leq r \qquad \text{and} \qquad
                                     B(y,r/M) \subset \Omega.
    \end{align*}
  \item We say that~$\Omega$  satisfies the \emph{exterior corkscrew
      condition} if $\Rn \setminus \Omega$ satisfies the interior
    corkscrew condition.
  \end{enumerate}
\end{definition}
\begin{definition}[Harnack Chain Condition]\label{def:hcc}
  We say that~$\Omega \subset \Rn$ satisfies the \emph{(interior) Harnack
  chain condition} if there exist~$R>0$ and $M \in \setN$ such that
  for any~$\epsilon>0$, $r \in (0,R)$, $x \in \partial \Omega$
  and $y_1,y_2 \in B(x,r) \cap \Omega$ with
  $\abs{y_1 - y_2} \leq \epsilon 2^k$ and
  $d(y_j,\partial \Omega) \geq \epsilon$ for $j=1,2$ there exists a
  chain of~$M k$ balls $B_1, \dots, B_{Mk}$ in~$\Omega$
  connecting~$y_1$ and $y_2$ satisfying
  \begin{enumerate}
  \item $y_1 \in B_1$, $y_2 \in B_{Mk}$,
  \item $\frac 1M \ell(B_j) \leq d(B_j, \partial \Omega) \leq M
    \ell(B_j)$ for $j=1, \dots, Mk$,
  \item
    $\ell(B_j) \geq \frac 1M \min \bigset{ d(y_1, B_j), d(y_2,
      B_j)}$ for $j=1, \dots, Mk$.
  \end{enumerate}
\end{definition}
\begin{definition}[NTA domain] \label{def:NTA}
  We say that a domain~$\Omega \subset \Rn$ is an \emph{NTA} (non-tangentially
  accessible) domain if $\Omega$ satisfies the interior corkscrew
  condition, the exterior interior corkscrew condition and the
  interior Harnack chain condition.
\end{definition}

\begin{definition}\label{def:Al}
  We say that~$\Omega \subset \Rn$ has \emph{Ahlfors regular
    boundary} if there exists~$R>0$ and $M>0$ such that for
  all~$r \in (0,R)$
  \begin{align}
    \label{eq:10}
    \frac 1M  r^{n-1} \leq \mathcal{H}^{n-1}(B(x,r) \cap \partial
    \Omega) \leq M r^{n-1}.
  \end{align}
\end{definition}

In the following we tacitly require that our domains satisfy the following assumption:
\begin{assumption}\label{ass:main}
  We assume that~$\Omega$ satisfies the following assumptions:
  \begin{enumerate}
  \item $\Omega$ is an NTA domain.
  \item $\Omega$ has Ahlfors regular boundary.
  \end{enumerate}
\end{assumption}
Note that all Lipschitz graph domains satisfy this assumption.

Let us now construct families of balls that we will use later in the
construction of our traces:

For each~$j \in \setZ$, let $(B_{j,k})_k$ denote a (countable) cover of balls of~$\Rn$
with diameter~$\ell(B_{j,k})$ such that
\begin{enumerate}
\item $\frac 18 \cdot 2^{-j} \leq \ell(B_{j,k}) \leq \frac 14 \cdot
  2^{-j}$.
\item The scaled balls $(\frac 78 B_{j,k})_k$ cover~$\Rn$.
\item Each family~$(B_{j,k})_k$ is locally finite with covering
  constant independent of~$j$, i.e.
  \begin{align*}
    \sup_j \sum_k \chi_{B_{j,k}} \leq c.
  \end{align*}
\end{enumerate}
For each~$j$ let~$(\eta_{j,k})_k$ be a partition of unity with
respect to the $(B_{j,k})_k$ such that for all~$j,k$
\begin{align}
  \label{eq:rho}
  \norm{\eta_{j,k}}_{L^\infty} + \ell(B_{j,k}) 
  \norm{\nabla \eta_{j,k}}_{L^\infty} &\leq c.
\end{align}
Now, we define the $2^{-j}$-neighbourhood~$U_j$ of~$\partial \Omega$ by
\begin{align*}
  U_j &:= \set{ x \in \Omega\,:\, d(x,\partial \Omega) < 2^{-j}}.
\end{align*}
Since $\Omega$ satisfies the interior corkscrew condition, we can
find for each ball~$B_{j,k}$ close to the boundary a \emph{reflected
  ball} $B_{j,k}^\sharp$ close by. We will use these reflected balls later to
define the local projections of our functions. More precisely:
\begin{enumerate}[label={(B\arabic{*})},start=1]
\item \label{itm:B1} There exists~$j_0 \in \setZ$, such that the
  following holds: For each~$B_{j,k}$ with $j \geq j_0$ and
  $B_{j,k} \cap U_j \neq \emptyset$, there exists a
  ball~$B_{j,k}^\sharp \subset \Omega$ with
  $\ell(B_{j,k}^\sharp) \eqsim \ell(B_{j,k}) \eqsim
  d(B_{j,k}^\sharp, \partial \Omega)$ and $d(B_{j,k},B_{j,k}^\sharp)\lesssim \ell(B_{j,k})$,
  where the hidden constants are independent of~$j,k$.
\end{enumerate}
Moreover, due to the Harnack chain condition we can
connect two reflected balls  of neighbouring balls by a small chain of
balls. More precisely, we have the following. 
\begin{enumerate}[label={(B\arabic{*})},start=2]
\item \label{itm:B2} If $B_{j,k}\subset\Omega$ and $j \geq j_0$, then there exists a chain of
  balls~$W_1,\dots, W_\gamma$ with $\gamma$ uniformly bounded, such
  that
  \begin{enumerate}
  \item $W_1 = B_{j,k}$ and $W_\gamma = B_{j,k}^\sharp$;
  \item  $\abs{W_\beta \cap W_{\beta+1}} \eqsim \abs{W_\beta} \eqsim
    \abs{W_{\beta+1}} \eqsim \abs{B_{j,k}}$ for $\beta=1, \dots,\gamma-1$;
  \item $ \ell(W_\beta) \eqsim
    \ell(B_{j,k})$ for $\beta=1,\dots, \gamma$;
  \end{enumerate}
  The hidden constants are independent of~$j,k,\beta$.  

  We define
  $\Omega(B_{j,k},B^\sharp_{j,k}) := \bigcup_{\beta=1}^\gamma
  W_\beta$.
\item \label{itm:B3} If $B_{j,k} \cap B_{l,m} \neq \emptyset$ and
  $j,l \geq j_0$ with $\abs{j-l} \leq 1$, then there exists a chain of
  balls~$W_1,\dots, W_\gamma$ with $\gamma$ uniformly bounded, such
  that
  \begin{enumerate}
  \item $W_1 = B_{j,k}^\sharp$ and $W_\gamma = B_{l,m}^\sharp$;
  \item  $\abs{W_\beta \cap W_{\beta+1}} \eqsim \abs{W_\beta} \eqsim
    \abs{W_{\beta+1}} \eqsim \abs{B_{j,k}}$ for $\beta=1, \dots,\gamma-1$;
  \item $d(W_\beta, \partial \Omega) \eqsim \ell(W_\beta) \eqsim
    \ell(B_{j,k})$ for $\beta=1,\dots, \gamma$;
  \end{enumerate}
  The hidden constants are independent of~$j,k,\beta$.

  We define
  $\Omega(B^\sharp_{j,k},B^\sharp_{l,m}) := \bigcup_{\beta=1}^\gamma
  W_\beta$.
\end{enumerate}
By construction of the chains above, we get:
\begin{enumerate}[label={(B\arabic{*})},start=4]
\item  \label{itm:B4}
  There exists~$k_0 \geq 2$ such that the following holds uniformly
  in~$j \geq j_0$
  \begin{gather*}
    \sum_{m\,:\, B_{j,m} \cap U_j \neq \emptyset}
    \chi_{B^\sharp_{j,m}} \leq c\, \chi_{U_{j-k_0} \setminus U_{j +
                            k_0}},
    \\
    \sum_{m\,:\, B_{j,m} \cap U_j \neq \emptyset} \sum_{k\,:\, B_{j+1,k} \cap B_{j,m} \neq
      \emptyset} \chi_{\Omega(B^\sharp_{j,m},B^\sharp_{j+1,k})} \leq c\,
    \chi_{U_{j-k_0} \setminus U_{j + k_0}}.  
  \end{gather*}
\end{enumerate}

\subsection{Trace operator}
\label{sec:trace-operator}

We will now construct the trace operator of~$\BVA(\Omega)$. We will
obtain the traces by a suitable approximation process. In particular,
we will define truncations~$T_j u$ which are smooth close to the boundary and admit
classical traces. The limits will later provide our trace.

We define
\begin{align*}
  \Pi_{j,k} u := \Pi_{B_{j,k}^\sharp} u.
\end{align*}
Let $\rho_j \in C^\infty(\Omega)$ be such that
$\chi_{U_{j+1}} \leq \rho_j \leq \chi_{U_j}$ and
$\norm{\nabla \rho_j}_\infty \lesssim 2^j$ and let
$u \in \BVA(\Omega)$. Then for~$j \geq j_0$ we define $T_j u$
in~$\Omega$ by
\begin{align}
  \label{eq:def-Tj}
  T_j u &:= u - \rho_j \sum_k \eta_{j,k} \big(u - \Pi_{j,k} u\big) =
          (1-\rho_j) u + \rho_j \sum_k \eta_{j,k} \Pi_{j,k} u.
\end{align}
Due to the support of~$\eta_{j,k}$ the sum in the definition is
locally finite. In particular, the sum is well defined
in~$L^1_{\loc}(\Omega)$. The function~$T_ju$ is an approximation
of~$u$, that replaces the values of~$u$ in the neighborhood
of~$\partial \Omega$ of distance~$2^{-j}$ by local
averages. These averages are performed slightly inside the domain on
the balls~$B_{j,k}^\sharp$.

We begin with an auxiliary estimate involving $\Pi_{j,k} u$.
\begin{lemma}
  \label{lem:Pjk-est}
We have the following estimates:
  \noindent%
  \begin{enumerate}
  \item \label{itm:Pjk-est1} There holds
    \begin{align*}
      \norm{\Pi_{j,k} u}_{L^\infty(B_{j,k})} 
      &\lesssim
        \dashint_{B_{j,k}^\sharp} \abs{u}\,dx.
    \end{align*}
  \item \label{itm:Pjk-est2} If
    $B_{j,m} \cap (U_j \setminus U_{j+2}) \neq \emptyset$, then
    $B_{j,m} \subset \Omega$ and
    \begin{align*}
      \norm{u-\Pi_{j,m} u}_{L^1(B_{j,m})} 
      &\lesssim \ell(B_{j,m})
        \abs{\opA u}\big(\Omega(B_{j,m},B^\sharp_{j,m})\big).
    \end{align*}
  \item \label{itm:Pjk-est3} If
    $B_{j+1,k} \cap B_{j,m} \neq \emptyset$, then
    \begin{align*}
      \abs{B_{j,m}}\, \norm{\Pi_{j+1,k} u - \Pi_{j,m} u}_{L^\infty(B_{j,m})} 
      &\lesssim \ell(B_{j,m})
        \abs{\opA u}\big(\Omega(B^\sharp_{j+1,k},B^\sharp_{j,m})\big) .
    \end{align*}
  \end{enumerate}
\end{lemma}
\begin{proof}
\noindent
  \begin{enumerate}
  \item Since $\Pi_{j,k}$ maps to $N(\opA)$ and $N(\opA) \subset
    \mathscr{P}_l$, this is just the usual inverse estimate for
    polynomials of a fixed degree.
  \item The definition of $U_j$ and $\ell(B_{j,m}) \leq \frac 14 2^{-j}$
    implies $B_{j,m} \subset \Omega$. 
We compute
\begin{align*}
 & \norm{u-\Pi_{j,m} u}_{L^1(B_{j,m})} = \norm{u-\Pi_{B^\sharp_{j,m}} u}_{L^1(B_{j,m})} \\&\leq\norm{u-\Pi_{B_{j,m}} u}_{L^1(B_{j,m})}+\norm{\Pi_{B_{j,m}} u-\Pi_{B^\sharp_{j,m}} u}_{L^1(B_{j,m})}.
\end{align*}
The first term can be estimated by Poincar\'e's inequality from Theorem~\ref{thm:poincare} which yields immediately
\begin{align*}
\norm{u-\Pi_{B_{j,m}} u}_{L^1(B_{j,m})}\lesssim\ell(B_{j,m})
        \abs{\opA u}\big(B_{j,m}\big).
\end{align*}
For the second term we make use of the Harnack chain conditions (recall Definition~\ref{def:hcc}) and, using \ref{itm:B2}, connect $B_{j,m}$ and $B^\sharp_{j,m}$ by a chain
\begin{align*}
\Omega(B_{j,k},B^\sharp_{j,m})=\bigcup_{\beta=1}^\gamma W_\beta,
\end{align*}
where $W_1,...,W_\gamma$ are balls of size proportional to $\ell(B_{j,m})$. In particular, we have $W_1=B_{j,m}$ and $W_\gamma=B^\sharp_{j,m}$. Moreover, we can assume that $|W_\beta\cap W_{\beta+1}|\eqsim |W_\beta|\eqsim \ell(B_{j,m})$ for all $\beta$.
 Now, we gain
\begin{align*}
\norm{\Pi_{B_{j,m}} u-\Pi_{B^\sharp_{j,m}} u}_{L^1(B_{j,m})}&\leq \sum_{\beta=1}^{\gamma-1}\norm{\Pi_{W_{\beta+1}} u-\Pi_{W_\beta} u}_{L^1(B_{j,m})}\\
&\lesssim \sum_{\beta=1}^{\gamma-1}\norm{\Pi_{W_{\beta+1}} u-\Pi_{W_\beta} u}_{L^1(W_{\beta+1}\cap W_\beta)}\\
&\lesssim \sum_{\beta=1}^{\gamma}\norm{u-\Pi_{W_\beta} u}_{L^1(W_\beta)}
\end{align*}
using equivalence of norms on $N(\A)$. Finally, using again Theorem~\ref{thm:poincare} in conjunction with~\ref{itm:B4}, 
\begin{align*}
\norm{\Pi_{B_{j,m}} u-\Pi_{B^\sharp_{j,m}} u}_{L^1(B_{j,m})}
&\lesssim \ell(B_{j,m})\sum_{\beta=1}^{\gamma}\abs{\opA u}\big(W_\gamma\big)\\&\lesssim\ell(B_{j,m})
        \abs{\opA u}\big(\Omega(B_{j,m},B^\sharp_{j,m})\big).
\end{align*}
Gathering estimates, we arrive at the claim.
  \item First, by the inverse estimate for polynomials, we have
    \begin{align*}
      \abs{B_{j,m}}\, \norm{\Pi_{j+1,k} u - \Pi_{j,m} u}_{L^\infty(B_{j,m})} 
      &\lesssim \norm{\Pi_{j+1,k} u - \Pi_{j,m} u}_{L^1(B_{j,m})} \\
&=\norm{\Pi_{B^\sharp_{j+1,k}} u - \Pi_{B^\sharp_{j,m}} u}_{L^1(B_{j,m})}.
    \end{align*}
Now, connecting $B^\sharp_{j+1,k}$ and $B^\sharp_{j,m}$ via the chain
$\Omega(B^\sharp_{j+1,k},B^\sharp_{j,m})$ (recall~\ref{itm:B3}), we obtain the claim arguing exactly as in b).
  \end{enumerate}
\end{proof}
The following lemma shows that~$T_j$ is well defined on~$L^1(\Omega)$.
\begin{lemma}
  $T_j\,:\, L^1(\Omega) \to L^1(\Omega)$ is linear and bounded.
\end{lemma}
\begin{proof}
  We estimate pointwise on~$\Omega$
  \begin{align}
    \label{eq:Tju-pointwise}
    \begin{aligned}
      \abs{T_j u} &\leq (1-\rho_j) \abs{u} + \rho_j \sum_k
      \chi_{B_{j,k}} \norm{\Pi_{j,k} u}_{L^\infty(B_{j,k})}.
    \end{aligned}
  \end{align}
  With Lemma~\ref{lem:Pjk-est} we get
  \begin{align*}
    \abs{T_j u}  &\lesssim \chi_{\Omega \setminus U_{j+1}} \abs{u}+ \sum_{k\,:\,
        B_{j,k} \cap U_j \neq \emptyset} \chi_{B_{j,k}}
      \dashint_{B_{j,k}^\sharp} \abs{u}\,dx.
  \end{align*}
  This implies
  \begin{align*}
    \norm{T_j u}_{L^1(\Omega)}  
    &\lesssim \norm{u}_{L^1(\Omega \setminus U_{j+1})} + \sum_{k\,:\,
      B_{j,k} \cap U_j \neq \emptyset} \abs{B_{j,k}}
      \dashint_{B_{j,k}^\sharp} \abs{u}\,dx
    \\
    &\lesssim \norm{u}_{L^1(\Omega \setminus U_{j+1})} + \sum_{k\,:\,
      B_{j,k} \cap U_j \neq \emptyset}
      \int_{B_{j,k}^\sharp} \abs{u}\,dx.
  \end{align*}
  Since the $B^\sharp_{j,k}$ are locally finite by~\ref{itm:B4}, we get
  $\norm{T_j u}_{L^1(\Omega)} \lesssim \norm{u}_{L^1(\Omega)}$ as
  desired.
\end{proof}
The next two lemmas show now that $T_{j+1} u - T_j u$ is summable
in~$L^1(\Omega)$ and $\BVA(\Omega)$.
\begin{lemma}
  \label{lem:TdiffL1}
  Let $u \in L^1(\Omega)$ and $j \geq j_0$. Then
  \begin{align*}
    \norm{T_{j+1} u - T_j u}_{L^1(\Omega)} &\lesssim
                                             \norm{u}_{L^1(U_{j+1-k_0}
                                             \setminus U_{j+k_0})}.
  \end{align*}
\end{lemma}
\begin{proof}
  Let $j \geq j_0$. Then we have 
  \begin{align*}
    T_{j+1} u - T_j u 
    &= (\rho_j-\rho_{j+1}) u + \rho_{j+1} \sum_k \eta_{j+1,k} \Pi_{j+1,k}
      u - \rho_j \sum_m \eta_{j,m} \Pi_{j,m} u. 
  \end{align*}
  Now
  \begin{align*}
    \norm{(\rho_j-\rho_{j+1}) u}_{L^1(\Omega)} &\le
                                                 \norm{u}_{L^1(U_j
                                                 \setminus U_{j+2})}.
  \end{align*}
  Moreover, by Lemma~\ref{lem:Pjk-est} (a) it follows that 
  \begin{align*}
    \norm{\rho_j \eta_{j,m} \Pi_{j,m} u}_{L^1(\Omega)} 
    &\le c
      \abs{B_{j,m}} \norm{\Pi_{j,m}
      u}_{L^\infty(B_{j,m})} 
      \leq c\, 
      \norm{u}_{L^1(B^\sharp_{j,m})},
  \end{align*}
  where it suffices to consider those $j$ with
  $B_{j,m} \cap U_j \neq \emptyset$.  Now~\ref{itm:B4} implies
  \begin{align*}
    \sum_m \norm{\rho_j \eta_{j,m} \Pi_{j,m}
    u}_{L^1(\Omega)}& \leq 
                      c\, \norm{u}_{L^1(U_{j-{k_0}}\setminus U_{j+{k_0}})}.
  \end{align*}
  Analogously,
  \begin{align*}
    \sum_k \norm{\rho_j \eta_{j+1,k} \Pi_{j+1,k}
    u}_{L^1(\Omega)}& \leq 
                      c\, \norm{u}_{L^1(U_{j+1-{k_0}}\setminus
                      U_{j+1+k_0})}.
  \end{align*}
  Combining the above estimates proves the lemma.
\end{proof}
\begin{lemma}
  \label{lem:TdiffA}
  Let $u \in \BVA(\Omega)$ and $j \geq j_0$. Then
  \begin{align*}
    \norm{\opA(T_{j+1} u - T_j u)}_{L^1(\Omega)} 
    &\lesssim \abs{\opA u}(U_{j-k_0} \setminus U_{j+k_0}).
  \end{align*}
\end{lemma}
\begin{proof}
  Using that $\sum_m \eta_{j,m} = \sum_k \eta_{j+1,k} =1$ in~$\Omega$
  we get
  \begin{align}
    \label{eq:9}
    \begin{aligned}
      T_{j+1} u - T_j u &= (\rho_j-\rho_{j+1}) \sum_m \eta_{j,m} (u -
      \Pi_{j,m} u)
      \\
      &\quad + \rho_{j+1} \sum_{k,m} \eta_{j+1,k} \eta_{j,m}
      (\Pi_{j+1,k} u - \Pi_{j,m} u)
      \\
      &=: I + II.
    \end{aligned}
  \end{align} 
In order to estimate $\norm{\opA(T_{j+1} u - T_j u)}_{L^1(\Omega)} $
it is crucial that $\opA\Pi_{j+1,k} u =\opA \Pi_{j,m} u=0$ and the gradients of
$\rho_j,\rho_{j+1}, \eta_{j,m}$ and $\eta_{j+1,k}$ are bounded by $2^j$, recall \eqref{eq:rho}.
  Let us consider~$II$. We only have to estimate those summands with~$k,m$ satisfying
  $B_{j+1,k} \cap B_{j,m} \neq \emptyset$ since otherwise
  $\eta_{j+1,k} \eta_{j,m}=0$. For each such~$k,m$ we estimate
  the~$L^1(\Omega)$-norm of $\opA II$ by
  Lemma~\ref{lem:Pjk-est}~\ref{itm:Pjk-est3}. Now, in combination
  with~\ref{itm:B4} we get
  \begin{align*}
    \|\opA II\|_{L^1(\Omega)} &\lesssim \abs{\opA u}(U_{j-k_0} \setminus U_{j+k_0}).
  \end{align*}
  Let us consider~$I$. We only need to estimate those summands
  with~$m$ satisfying
  $B_{j,m} \cap (U_j \setminus U_{j+2}) \neq \emptyset$, since
  otherwise $(\rho_j-\rho_{j+1}) \eta_{j,m}=0$. For each such~$m$ we
  estimate the~$L^1(\Omega)$-norm of $\opA I$ by
  Lemma~\ref{lem:Pjk-est}~\ref{itm:Pjk-est2}. Now, in combination
  with~\ref{itm:B4} we get
  \begin{align*}
    \|\opA I\|_{L^1(\Omega)} &\lesssim \abs{\opA u}(U_{j-k_0} \setminus U_{j+k_0}).
  \end{align*}
  The proof is complete.
\end{proof}
Based on the two lemmas above, we now study the convergence $T_j u \to u$.
\begin{corollary}
  \label{cor:TjlimitL1BVA}
  If $u \in L^1(\Omega)$, then
  \begin{align}
    \label{eq:TjlimitL1BVA}
    u &= T_{j_0} u + \sum_{l=j_0}^\infty \big( T_{l+1} u - T_l u\big)
        = \lim_{j \to \infty} T_j u
  \end{align}
  in $L^1(\Omega)$. If additionally $u \in \BVA(\Omega)$,
  then~\eqref{eq:TjlimitL1BVA} also holds in $\BVA(\Omega)$.
\end{corollary}
\begin{proof}
  Since $\rho_j \to 0$ in $L^1_{\loc}(\Omega)$, it is clear
  that~$T_j u \to u$ in $L^1_{\loc}(\Omega)$.

  Note that for~$j \geq j_0$
  \begin{align}
    \label{eq:11}
    T_j u &= T_{j_0} u + \sum_{l=j_0}^{j-1} \big( T_{l+1} u - T_l u\big) 
  \end{align}
  It follows from Lemma~\ref{lem:TdiffL1} and Lemma~\ref{lem:TdiffA}
  that $T_{l+1} u - T_l u$ are summable in~$L^1(\Omega)$,
  resp. in~$\BVA(\Omega)$, since the $U_{j+1-k_0} \setminus U_{j+k_0}$
  are locally finite with respect to~$j$. Hence, $T_j u$ is a Cauchy sequence
  in~$L^1(\Omega)$, resp. in~$\BVA(\Omega)$.  Since the limit must
  agree with the~$L^1_{\loc}(\Omega)$ limit, which is~$u$, the claim
  follows.
\end{proof}
Since $T_ju$ is smooth close to the boundary~$\partial\Omega$, it is
possible to evaluate the classical trace~$\trace(T_j u)$. We now
show that these traces form a $L^1(\partial \Omega)$-Cauchy sequence.
\begin{lemma}
  \label{lem:trace-cauchy}
  Let $u \in \BVA(\Omega)$. Then
  \begin{align*}
    \norm{\trace(T_{j+1} u) - \trace(T_j u)}_{L^1(\partial \Omega)}
    &\lesssim 
      \abs{\opA u}(U_{j-k_0} \setminus U_{j + k_0})
      \\
    \intertext{and}
    \bignorm{\trace(T_{j_0} u)}_{L^1(\partial \Omega)} 
    &\lesssim 2^{j_0} \norm{u}_{L^1(U_{j_0-k_0}\setminus U_{j_0 + k_0})}.
  \end{align*}
\end{lemma}
\begin{proof}
  We begin with the first estimate. It follows from~\eqref{eq:9} that
  \begin{align*}
    \trace(T_{j+1} u) - \trace(T_j u) 
    &=  \sum_{k,m} \trace \big(\eta_{j+1,k} \eta_{j,m}
      (\Pi_{j+1,k} u - \Pi_{j,m} u) \big),
  \end{align*}
  where the sums are locally finite sums.  Hence,
  \begin{align*}
    \norm{\trace(T_{j+1} u) - \trace(T_j u)}_{L^1(\partial \Omega)}
    &\leq \sum_{k,m} \bignorm{\trace(\eta_{j+1,k} \eta_{j,m}
      (\Pi_{j+1,k} u - \Pi_{j,m} u))}_{L^1(\partial \Omega)}.
  \end{align*}
  We only have to consider those~$k,m$ with $B_{j+1,k} \cap B_{j,m}
  \neq \emptyset$. For such~$k,m$
  \begin{align*}
    \lefteqn{\bignorm{\trace \big(\eta_{j+1,k} \eta_{j,m}
    (\Pi_{j+1,k} u - \Pi_{j,m} u) \big)}_{L^1(\partial \Omega)}} \qquad 
    &
    \\
    &\leq
      \bignorm{\Pi_{j+1,k} u - \Pi_{j,m} u}_{L^\infty(B_{j,m})} \,
      \mathcal{H}^{n-1} ( \partial \Omega \cap B_{j+1,k} \cap 
      B_{j,m}). 
  \end{align*}
  We estimate the first factor by
  Lemma~\ref{lem:Pjk-est}~\ref{itm:Pjk-est3} and the second by the
  Ahlfors regularity of the boundary, see~\eqref{eq:10}, and thereby obtain
  \begin{align*}
    \bignorm{\trace \big(\eta_{j+1,k} \eta_{j,m}
    (\Pi_{j+1,k} u - \Pi_{j,m} u) \big)}_{L^1(\partial \Omega)}
    &\lesssim
      \abs{\opA
      u}\big(\Omega(B^\sharp_{j+1,k},B^\sharp_{j,m})\big).
  \end{align*}
  Summing over~$k$ and $m$ and using~\ref{itm:B4} implies
  \begin{align*}
    \norm{\trace(T_{j+1} u) - \trace(T_j u)}_{L^1(\partial \Omega)}
    &\lesssim 
      \abs{\opA u}(U_{j-k_0} \setminus U_{j + k_0}).
  \end{align*}
  This proves the first estimate.

  Let us now estimate $\norm{\trace(T_{j_0})}_{L^1(\partial
    \Omega)}$. We begin with
  \begin{align*}
    \trace(T_{j_0}) &= \sum_k \trace\big(\eta_{j_0,k} \Pi_{j_0,k} u\big).
  \end{align*}
  For each~$k$ with $B_{j_0,k} \cap \partial \Omega$ there holds
  \begin{align*}
    \bignorm{\trace\big(\eta_{j_0,k} \Pi_{j_0,k}
    u\big)}_{L^1(\partial \Omega)} 
    &\leq \norm{\Pi_{j_0,k} u}_{L^\infty(B_{j_0,k})}
      \mathcal{H}^{n-1}(\partial \Omega \cap B_{j_0,k}).
  \end{align*}
  We estimate the first factor by
  Lemma~\ref{lem:Pjk-est}~\ref{itm:Pjk-est1} and the second by the
  Ahlfors regularity of the boundary, see~\eqref{eq:10}. This gives
  \begin{align*}
    \bignorm{\trace\big(\eta_{j_0,k} \Pi_{j_0,k}
    u\big)}_{L^1(\partial \Omega)} 
    &\lesssim \frac{1}{\ell(B_{j_0})} \int_{B^\sharp_{j_0,k}} \abs{u}\,dx.
  \end{align*}
  Summing over~$k$ and $m$ and using~\ref{itm:B4} implies
  \begin{align*}
    \bignorm{\trace(T_{j_0} u)}_{L^1(\partial \Omega)} 
    &\lesssim 2^{j_0} \norm{u}_{L^1(U_{j_0-k_0}\setminus U_{j_0 + k_0})}.
  \end{align*}
  This proves the claim.
\end{proof}  
Recall that by Corollary~\ref{cor:TjlimitL1BVA} we have
\begin{align*}
  u &= T_{j_0} u + \sum_{l=j_0}^\infty \big( T_{l+1} u - T_l u\big)
      = \lim_{j \to \infty} T_j u
\end{align*}
in $\BVA(\Omega)$. Morover, Lemma~\ref{lem:trace-cauchy} shows that
\begin{align*}
  \trace(T_{j_0}u) + \sum_{j \geq j_0} \big(\trace(T_{j+1} u) -
  \trace(T_j u)\big) = \lim_{j \to \infty} \trace(T_j(u)).
\end{align*}
is well defined in~$L^1(\partial \Omega)$. Finally,
\begin{align*}
  \bignorm{\lim_{j \to \infty} \trace(T_j(u))}_{L^1(\partial \Omega)} 
  &\leq
    \bignorm{\trace(T_{j_0}(u))}_{L^1(\partial \Omega)} + \sum_{j
    \geq j_0} \bignorm{\trace(T_{j+1} u) -
    \trace(T_j u)}_{L^1(\partial \Omega)}
  \\
  &\lesssim 2^{j_0} \norm{u}_{L^1(U_{j_0-k_0}\setminus U_{j_0 + k_0})}
    + \sum_{j
    \geq j_0}    \abs{\opA u}(U_{j-k_0} \setminus U_{j + k_0}) 
  \\
  &\lesssim \norm{u}_{L^1(\Omega)}
    + \abs{\opA u}(\Omega)
\end{align*}
by Lemma \ref{lem:trace-cauchy}.
This allows us to define for every~$u \in \BVA(\Omega)$ a trace
\begin{align}
  \label{eq:def-tracenew}
  \tracenew(u) &:= \lim_{j \to \infty} \trace(T_j u),
\end{align}
the limit being understood in the $\lebe^{1}(\partial\Omega)$-sense. This limit satisfies  
\begin{align}
  \label{eq:est-tracenew}
  \bignorm{\tracenew(u)}_{L^1(\partial \Omega)} 
  &\lesssim \norm{u}_{L^1(\Omega)}
    + \abs{\opA u}(\Omega).
\end{align}
We now show that~$\tracenew$ coincides with~$\trace$ for all smooth
functions and hence start with an approximation result.
\begin{lemma}
  \label{lem:TjlimitC0}
  Let $u \in C^0(\overline{\Omega})$ be uniformly continuous. Then
  $T_j u \to u$ in~$C^0(\overline{\Omega})$.
\end{lemma}
\begin{proof}
  We have
  \begin{align*}
    u-T_j u &= \rho_j \sum_k \eta_{j,k} (u - \Pi_{j,k} u),
  \end{align*}
  where it suffices to take the sum over those~$k$ with
  $B_{j,k} \cap U_j \neq \emptyset$. Let us take one of those~$k$. We
  will show that
  $\norm{\eta_{j,k} (u- \Pi_{j,k} u)}_{L^\infty(\Omega)}$ will be
  small for large~$j$. Since the~$B_{j,k}$ are locally finite with
  respect to~$k$ (with a covering number independent of~$j$), this
  will prove the lemma.

  Since~$\opA$ maps constants to zero, the projections~$\Pi_{j,k}$ map
  constants to themselves. Let
  $\mean{u}_{B_{j,k}^\sharp} := \dashint_{B_{j,k}^\sharp}u\,dx$, then
  with Lemma~\ref{lem:Pjk-est}~\ref{itm:Pjk-est1}
  \begin{align*}
    \norm{\eta_{j,k} (u- \Pi_{j,k} u)}_{L^\infty(B_{j,k})} 
    &\leq 
      \norm{\smash{u- \mean{u}_{B_{j,k}^\sharp}}}_{L^\infty(B_{j,k})} +
      \norm{\Pi_{j,k}(u-\mean{u}_{B_{j,k}^\sharp})}_{L^\infty(B_{j,k})}
    \\
    &\lesssim 
      \norm{u- \mean{u}_{B_{j,k}^\sharp}}_{L^\infty(B_{j,k})} +
      \dashint_{B_{j,k}^\sharp} \abs{u-\mean{u}_{B_{j,k}^\sharp}}\,dx.
  \end{align*}
  Since~$u$ is uniformly continuous, the~$B_{j,k}$ and
  $B_{j,k}^\sharp$ are small and close to each other, cf. \ref{itm:B1}, we see that both
  expressions on the right-hand side are small for large~$j$ uniformly
  in~$k$. The concludes the proof.
\end{proof}
\begin{corollary}\label{cor:classical}
  Let $u \in \BVA(\Omega) \cap C^0(\overline{\Omega})$ be uniformly
  continuous. Then $\tracenew(u) = \trace(u)$.
\end{corollary}
\begin{proof}
  We see from Corollary~\ref{cor:TjlimitL1BVA} and
  Lemma~\ref{lem:TjlimitC0} that $T_j u \to u$ in $\BVA(\Omega)$ and
  in~$C^0(\overline{\Omega})$. By definition of~$\tracenew(u)$, we have
  $\trace (T_j u) \to \tracenew(u)$. Since $T_j u \to u$
  in~$C^0(\overline{\Omega})$, we also have
  $\trace(T_j u) \to \trace(u)$ in~$C^0(\partial \Omega)$. The
  limits must agree in~$L^1_{\loc}(\partial \Omega)$,
  so~$\tracenew(u) = \trace(u)$.
\end{proof}
We have already seen that
$\tracenew\,:\, \BVA(\Omega) \to L^1(\partial \Omega)$ is continuous
with respect to the norm topology. We wish to use this to conclude
that~$\tracenew$ is the only extension of the classical trace
to~$\BVA(\Omega)$.  However, as smooth functions are not dense
in~$\BVA$ with respect to the norm topology, we switch to strict
convergence as in the $\bv$--case.
\begin{lemma}
  \label{lem:trace-strict}
  The trace
  operator~$\tracenew\,:\, \BVA(\Omega) \to L^1(\partial \Omega;\R^{N})$ is
  continuous with respect to the strict convergence of~$\BVA(\Omega)$.
\end{lemma}
\begin{proof}
  Let $u,u_k\in\BVA(\Omega)$ with $u_k \strictto u$ and $m\in\mathbb N$. 
  
  It follows from the definition~\eqref{eq:def-Tj} of~$T_j$ that
  for~$j>m+k_0$, there holds for all~$v \in \BVA(\Omega)$
  \begin{align*}
    T_j (\rho_m v) = \rho_m\, T_j v.
  \end{align*}
  Indeed, $\rho_m=1$ on the $B_{j,k}$ and the $B_{j,k}^\sharp$ for
  all~$m$ that contribute to the sum in~\eqref{eq:def-Tj}.

  This implies that
  \begin{align*}
    \tracenew(v) = \lim_{j\to \infty} \trace(T_j v)
    = \lim_{j\to \infty} \trace(T_j (\rho_m v)) =
    \tracenew(\rho_m v)\qquad\text{in}\;\lebe^{1}(\partial\Omega).
  \end{align*}
  Now, for all $k\in\mathbb{N}$,
  \begin{align*}
    \norm{\tracenew(u_k-u)}_{L^1(\partial \Omega)}
    & =
      \norm{\tracenew(\rho_m(u_k-u))}_{L^1(\partial \Omega)}.
  \end{align*}
  Thus, by~\eqref{eq:est-tracenew}
  \begin{align*}
    \norm{\tracenew(u_k-u)}_{L^1(\partial \Omega)}
    &\lesssim
      \norm{\rho_m(u_k - u)}_{L^1(\Omega)}+
      \abs{\opA
      (\rho_m(u_k-u))}(\Omega)
    \\
    &\lesssim
      \norm{u_k - u}_{L^1(\Omega)}+
      \abs{\opA u_k}(U_m) +       \abs{\opA u_k}(U_m) + 2^{-m}
      \norm{u_k -u}_{L^1(U_m)}. 
  \end{align*}
  Now, let $k,l \to \infty$. Since $u_k \strictto u$ in~$\BVA(\Omega)$
  and~$U_m$ is open, we get
  \begin{align*}
    \norm{\tracenew(u_k-u)}_{L^1(\partial \Omega)}
    &\lesssim
      \abs{\opA u}(U_m).
  \end{align*}
  The right-hand side converges to zero for~$m \to \infty$. Thus
  $\tracenew(u_k) \to \tracenew(u)$ in $L^1(\partial \Omega)$ for $k
  \to \infty$.
\end{proof}

In order to proceed, we need an smooth approximation result up to the
boundary in the
area-strict topology.
\begin{lemma}
  \label{lem:approx-COmega-bar}
Let $u \in \BVA(\Omega)$. Then there
  exists~$u_j \in C^\infty(\overline{\Omega})$ with $u_j \areato u$
  in $\BVA(\Omega)$.  
\end{lemma}
\begin{proof}
  For $j \geq j_0$ consider $T_j u$. Then $T_j u$ is $C^\infty$ in
  $\overline{U_{j+1}}$. Indeed, for all~$x \in U_{j+1}$ we have
  \begin{align*}
    (T_j u)(x) &= \sum_k \eta_{j,k} \Pi_{j,k} u.
  \end{align*}
  For each~$k$ with $B_{j,k} \cap U_{j+1} \neq 0$ we have 
  \begin{align*}
    \norm{\nabla (\eta_{j,k} \Pi_{j,k} u)}_\infty 
    &\lesssim \norm{\nabla \eta_{j,k}}_{L^\infty(B_{j,k})}
      \norm{\Pi_{j,k} u}_{L^\infty(B_{j,k})} + 
      \norm{\nabla \Pi_{j,k} u}_{L^\infty(B_{j,k})}.
  \end{align*}
  Using inverse estimates for polynomials and  Lemma~\ref{lem:Pjk-est}
  we get
  \begin{align*}
    \norm{\nabla (\eta_{j,k} \Pi_{j,k} u)}_\infty 
    &\lesssim \ell(B_{j,k}) \abs{B_{j,k}} 
      \norm{\Pi_{j,k} u}_{L^1(B_{j,k})} \lesssim 2^{j(n+1)}
      \norm{u}_{L^1(B^\sharp_{j,k})}.
  \end{align*}
  Hence, $T_j u$ is uniformly continuous on~$\overline{U}_{j+1}$.

  Now, let $\eta_\epsilon \colon\R^{n}\to\R$ be an standard mollifier
  (even and non-negative). It is well known that
  $u_{j,\epsilon} := \rho_{j+1} T_j u + ((1-\rho_{j+1}) T_j u) *
  \eta_\epsilon$
  converges to $T_j u$ as $\epsilon \searrow 0$ in $L^1(\Omega)$ as
  well as in the area-strict sense. Hence, we can
  find~$\epsilon_j$ such that
  \begin{align*}
    \norm{u_{j,\epsilon_j} - T_j u}_{L^1(\Omega)} &\leq 2^{-j},
    \\
    \bigabs{ \abs{\opA(T_j u)}(\Omega) - \abs{\opA(u_{j,\epsilon_j})}(\Omega)}
                                                  &\leq 2^{-j}.
  \end{align*}
  Moreover, recall that $T_ju \to u$ strongly in $\BVA(\Omega)$. This
  implies that $u_j := u_{j,\epsilon_j}$ has the desired
  property. This proves the strict convergence. The area-strict
  convergence follows by the same steps.
\end{proof}
As a consequence of Lemma~\ref{lem:trace-strict}
and Lemma~\ref{lem:approx-COmega-bar} we immediately obtain the following corollary. 
\begin{corollary}\label{cor:unique}
  The trace~$\tracenew\,:\, \BVA(\Omega)\to\lebe^{1}(\partial\Omega;\mathcal{H}^{n-1})$ is the unique
  strictly-continuous extension of the classical trace
  on~$\BVA(\Omega) \cap C^0(\overline{\Omega})$.
\end{corollary}

Due to the above results it is not anymore necessary to distinguish
the classical trace and our new trace. We collect our results proven so far in the following theorem. 
\begin{theorem}\label{thm:main1b}
  Let $\A$ be $\setC$-elliptic and let $\Omega$ be an NTA domain with Ahlfors regular boundary (see Assumption~\ref{ass:main}). Then there exists a trace
  operator~$\trace\,:\, \BVA(\Omega) \to L^1(\partial \Omega,
  \mathcal{H}^{n-1})$ such that the following holds:
  \begin{enumerate}
  \item \label{itm:main1-1} $\trace(u)$ conincides with the classical
    trace for all $u \in \BVA(\Omega) \cap C^0(\Omega)$.
  \item \label{itm:main1-2} $\trace(u)$ is the unique
    strictly-continuous extension of the classical trace
    on~$\BVA(\Omega) \cap C^0(\overline{\Omega})$.
  \item \label{itm:main1-3}
    $\trace(\WA(\Omega))=\trace(\BVA(\Omega))=L^1(\partial \Omega,
    \mathcal{H}^{n-1})$.
  \end{enumerate}
\end{theorem}
\begin{proof}
  The existence of $\trace$ is shown in Lemma~\ref{lem:trace-strict}.
Part a) follows from Corollary~\ref{cor:classical}, whereas b) is a consequence of Corollary~\ref{cor:unique}.
  Finally, the third part is a consequence of the fact that
  $\trace(\sobo^{1,1}(\Omega;\R^{N})=\lebe^{1}(\partial\Omega;\R^{N})$
  and $\sobo^{1,1}(\Omega;\R^{N})\subset\sobo^{\A,1}(\Omega)$. In
  particular, the sufficiency part of Theorem~\ref{thm:main1} is
  complete.
\end{proof}

\subsection{Necessity of $\setC$-ellipticity}\label{sec:necessity}

In this section we show that it is not possible to define an
$L^1$-trace of $\bvA$--functions if the operator $\A$ is not $\setC$-elliptic. As such, we extend the observation of Fuchs and Repin \cite{FuchsRepin} that $\mathbb{D}\ni z\mapsto 1/(z-1)\in\mathbb{C}$ is holomorphic and belongs to $\lebe^{1}(\mathbb{D};\mathbb{C})$ but does not belong to $\lebe^{1}(\partial\mathbb{D};\mathbb{C})$ (cp.~Example~\ref{ex:opexamp}(c)). 


\begin{theorem}[Without a Trace]
  \label{thm:no-trace}
  Suppose that~$\opA$ is not $\setC$-elliptic.  Let $B$ denote the
  unit ball of~$\Rn$. Then there exists a
  vector~$\xi_1 \in \Rn \setminus \set{0}$, such that for the half
  ball $B^{+}:=\{x\in B\colon\;\langle \xi_1,x\rangle > 0\}$ and the
  hyperplane $\mathfrak{H} := \set{x \in \Rn\,:\, \skp{\xi_1}{x} =0}$
  there exists a function $u \in \WA(B^{+}) \cap C^\infty(B^{+})$
  such that $u \notin L^1(\mathfrak{H} \cap B, \mathcal{H}^{n-1})$.
\end{theorem}
\begin{proof}
  We begin with the case that~$\opA$ is not $\setR$-elliptic.  Let us
  define $f(x_1,x_2) := (\abs{x_1} + \abs{x_2}^2)^{-\frac 34}$. The
  crucial observation now is that $f,\partial_2 f \in L^1(B)$.
  However, $f \notin L^1(\set{x_1=0}|_B, \mathcal{H}^{n-1})$. We have
  to adapt this example to our situation. Since~$\opA$ is not~$\setR$
  elliptic, there exists~$\xi_1\in \Rn \setminus \set{0}$ and
  $\eta_1 \in \RN \setminus \set{0}$ with $\opA[\xi_1]\eta_1 = 0$. We
  choose~$\xi_2,\dots, \xi_n$ such that $\xi_1,\dots, \xi_n$ is a
  basis. Now, define $\tau\,:\, \Rn \to \setR^2$ and
  $\sigma \,:\, \setR \to \RN$ by
  $\tau(x) := (\langle\xi_1,x\rangle,\langle\xi_2,x\rangle)$ and
  $\sigma(z) := z\,\eta_1$.  Moreover, we define
  $h_f\,:\, \Rn \to \RN$ by $h_f := \sigma \circ f \circ \tau$.  Then
  we obtain
  $(\opA h_f)(x) = \sum_{j=1}^2 \opA[\xi_j] \eta_1 (\partial_j
  f)(\tau(x))$ (compare~\eqref{eq:FDNchar1}).
  Since $\opA[\xi_1]\eta_1 = 0$, this simplifies to
  $(\opA h_f)(x) = \opA[\xi_2] \eta_1 (\partial_2 f)(\tau(x))$.  We
  choose the hyperplace~$\frH := \set{x\,:\, \skp{\xi_1}{x}=0}$. It
  follows from $f,\partial_2 f \in L^1(B)$ and
  $f \notin L^1(\set{x_1=0}|_B, \mathcal{H}^{n-1})$ that
  $u, \opA u \in L^1(B)$ and so in particular
  $u, \opA u\in\lebe^{1}(B^{+})$ with
  $B^{+}:=\{x\in B\colon\;\langle \xi_{1},x\rangle>0\}$ but
  $u \notin L^1(\frH \cap B, \mathcal{H}^{n-1})$. This concludes the
  proof in the case that~$\opA$ is not $\setR$-elliptic.
  
  Assume now that $\opA$ is $\setR$-elliptic but
  not~$\setC$-elliptic. Then as in Lemma~\ref{lem:FDNchar} there exist
  $\xi_1, \xi_2 \in \Rn$, resp. and $\eta_1, \eta_2 \in \Rn$, which
  are linearly independent such that
  $\opA[\xi_1 + i x_2](\eta_1 +i \eta_2)=0$. Define
  $f\,:\, \setC \to \setC$ by $f(z) := \frac 1z$. Then
  $f \in L^1(B_1)$ with $B_1 := \set{\abs{z}<1}$ but
  $f \notin L^1(\set{\Re(z)=0}|_{B_1}, \mathcal{H}^{n-1})$. As in
  Lemma~\ref{lem:FDNchar} we define $\tau\,:\, \Rn \to \setC$ and
  $\sigma\,:\, \setC \to \RN$ by
  $\tau(x) := \skp{\xi}{x} = \skp{\xi_1}{x} + i \skp{\xi_2}{x}$ and
  $\sigma(z) := \Re(z) \eta_1 - \Im(z) \eta_2$. Moreover, define
  $h_f \colon \Rn\to \RN$ by $h_f:= \sigma \circ f \circ \tau$. Then
  as in Lemma~\ref{lem:FDNchar} we have $(\opA h_f)(x) =0$ in
  $\mathcal{D}'(B^{+})$ with
  $B^{+}:=\{x\in B\colon \langle x_{1},x\rangle>0\}$. It follows from
  $f \in L^1(B^{+})$ and
  $f \notin L^1(\set{\Re(z)=0}|_{B_1}, \mathcal{H}^{n-1})$ that
  $h_f \in \WA(B)$ but
  $h_f \notin L^1(\frH \cap B, \mathcal{H}^{n-1})$ with
  $\frH := \set{x\,:\,\skp{\xi_1}{x}=0}$. This concludes the proof
  if~$\opA$ is $\setR$-elliptic but not $\setC$-elliptic.
\end{proof}
\begin{remark}
  \label{rem:no-trace}
  Theorem~\ref{thm:no-trace} shows the non-existence of a trace on
  some particular boundary hyperplane. If $\Omega$ does not enjoy this simple geometry but is a bounded domain with $\hold^{\infty}$--boundary, then we choose a boundary point $x_{0}\in\partial\Omega$ such that a suitable translation of the hyperplanes $\mathfrak{H}$ from the preceding proof becomes tangent to $\partial\Omega$ at $x_{0}$. In this situation, straightening the boundary locally around $x_{0}$ and applying the preceding theorem directly yield the non--existence of boundary traces in $\lebe^{1}(\partial\Omega;\mathcal{H}^{n-1})$. We leave the details to the reader. 
\end{remark}

\subsection{Gau{\ss}--Green Formula}
\label{sec:gauss-green}

In this section we deduce the Gau{\ss}-Green formula for functions
from~$\BVA(\Omega)$ which, with Theorem~\ref{thm:main1} at our
disposal, is a direct consequence of the Gau{\ss}--Green formula for
smooth functions. Let us note that up to here, only Assumption~\ref{ass:main} is required whereas in what follows we stick to a Lipschitz assumption\footnote{In principle, this can be weakened towards more general domains, but we will not need this in the sequel.} on $\partial\Omega$.
\begin{theorem}[Gau{\ss}-Green formula]
  \label{thm:gauss-green}
Let $\Omega\subset\mathbb R^n$ be open and bounded with Lipschitz boundary.
  For all $u \in \BVA(\Omega)$ and all~$\phi \in
  C^1(\overline{\Omega};\R^{N})$ we have
  \begin{align}
    \label{eq:gauss-green}
    \int_\Omega \opA u \cdot \phi\,dx 
    &=
      - \int_\Omega u \cdot \opA^* \phi\,dx + \int_{\partial \Omega} (\trace(u)
      \otimes_{\opA} \nu) \cdot  \phi\, d\mathcal{H}^{n-1},
  \end{align}
  where $\nu$ denotes the unit outer normal of~$\Omega$.
\end{theorem}
\begin{proof}
  Due to Lemma~\ref{lem:approx-COmega-bar}
  there exists a sequence $u_j \in C^\infty(\overline{\Omega})$ such
  that $u_j \strictto u$ in $\BVA(\Omega)$. Due to
  Lemma~\ref{lem:trace-strict} we also have $u_j \to u$ in
  $L^1(\partial \Omega, \mathcal{H}^{n-1})$. 
  Now,~\eqref{eq:gauss-green} is valid for each~$u_j$. Passing to the
  limit proves the claim.
\end{proof}
\begin{corollary}
  \label{cor:gluing}
  Let $\Omega \compactsubset U \subset \Rn$ such that $\Omega$ and $U$
  are open and bounded and have Lipschitz boundary. For
  $u \in \BVA(\Omega)$ and $v \in \BVA(U \setminus \Omega)$ define
  $w := \chi_\Omega u + \chi_{U \setminus \Omega} v$. Then
  $w \in \BVA(U)$ and
  \begin{align}
    \label{eq:jump}
    \opA w &= \opA u \mres_\Omega + \opA v\mres_{U
             \setminus \Omega} +
             (\trace^+(v) - \trace^-(u)) \otimes_\opA \nu
             \mathcal{H}^{n-1} \mres_{\partial \Omega},
  \end{align}
  where $\trace^+(u)$ denotes the interior trace of~$u$ and
  $\trace^-(v)$ denotes the exterior trace of~$v$ and~$\nu$ the unit outer
  normal of~$\Omega$.
\end{corollary}
\begin{proof}
  Let~$w$ be as given and let $\phi \in C^1_c(U)$. We split the
  domain~$U$ into $\Omega$ and $U \setminus \Omega$ and apply the
  Gau{\ss}-Green formula~\eqref{eq:gauss-green} first to~$U$ and then
  to $\Omega$ and $U \setminus \Omega$ separately. This yields
  \begin{align*}
 -\int_U w \cdot \opA^* \phi\,dx
    &= -\int_\Omega u \cdot \opA^* \phi\,dx
      - \int_{U\setminus \Omega} v \cdot \opA^* \phi\,dx
    \\
    &=
      \quad \int_\Omega \opA u \cdot \phi\,dx - \int_{\partial \Omega}
      (\trace^+(u)
      \otimes_{\opA} \nu) \cdot  \phi\, d\mathcal{H}^{n-1}
    \\
    &\quad +
      \int_{U \setminus \Omega} \opA v \cdot \phi\,dx + \int_{\partial \Omega}
      (\trace^+(v)
      \otimes_{\opA} \nu) \cdot  \phi\, d\mathcal{H}^{n-1}.
  \end{align*}
  This proves that $w \in \BVA(U)$ and the representation
  formula~\eqref{eq:jump}. 
\end{proof}

\subsection{Sobolev Spaces with Zero Boundary Values}
\label{sec:zero-boundary-values}

Using our trace operator, it is natural to define subspace of
functions with zero boundary values, i.e.
\begin{align*}
  \WA_0(\Omega) &:= \set{u \in \WA(\Omega)\,:\, \trace(u) =0}.
  \\
  \BVA_0(\Omega) &:= \set{u \in \BVA(\Omega)\,:\, \trace(u) =0}.
\end{align*}
However, in the context of Sobolev spaces~$\WA_0(\Omega)$ there are
two more variants to define these spaces. One by zero extension and one
by closure of~$C^\infty_c(\Omega)$. We will show below in
Theorem~\ref{thm:WA10} that all three definitions define the same spaces.

We begin with an auxiliary lemma which we need for $\WA_0(\Omega)$. For
slightly more generality we state it for $\BVA_0(\Omega)$.
\begin{lemma}
  \label{lem:BV-cut}
  Let $u \in \BVA_0(\Omega)$. Then $(1-\rho_j) u \to u$ in
  $\BVA(\Omega)$, with $\rho_j$ as in Section~\ref{sec:trace-operator}.
\end{lemma}
\begin{proof}
  We can assume that~$\Omega \compactsubset U \subset \Rn$ for some open,
  bounded~$U$ with Lipschitz boundary.
  By Corollary~\ref{cor:gluing} we can extend~$u$
  on~$U\setminus \Omega$ by zero. 

  We have
  \begin{align*}
    \opA((1-\rho_j)u-u) &= - \rho_j \opA u - u \otimes_\opA \nabla \rho_j.
  \end{align*}
  Hence,
  \begin{align*}
    \abs{\opA((1-\rho_j)u-u)}(\Omega) 
    &\leq \abs{\opA u}(U_j) + c\, r_j^{-1} \norm{u}_{L^1(U_j)}. 
  \end{align*}
  We will now show that
  \begin{align*}
    r_j^{-1} \norm{u}_{L^1(U_j)} &\lesssim \abs{\opA u}(U_{j-m})
  \end{align*}
  for some~$m \in \setN$ (and sufficiently large,
  i.e.~$j+m \geq j_0$). In fact, for fixed~$j$ define
  \begin{align*}
    K_j := \set{k \,:\, B_{j,k} \cap U_j \neq \emptyset}.
  \end{align*}
  By the geometry of~$\Omega$, we can find a factor~$\lambda>0$,
  such that for each~$k \in K_j$, the enlarged ball~$\lambda B_{j,k}$
  contains some ball~$B_{j,k}'$ that is completely in~$\Rn \setminus
  \Omega$. Now, for each~$k \in K_j$, we get by
  Theorem~\ref{thm:poincare2} 
  \begin{align*}
    \norm{u}_{L^1{(B_{j,k})}} 
    &\lesssim
      \norm{u}_{L^1{(\lambda B_{j,k})}} \lesssim r_j \abs{\opA u}(\lambda
      B_{j,k}) = r_j \abs{\opA u}(\Omega \cap \lambda B_{j,k}).
  \end{align*}
  Since the $(B_{j,k})_k$ are locally finite, so are the~$(\lambda
  B_{j,k})_k$. Now, if we choose~$m \in \setN$ such that $\Omega \cap
  \lambda B_{j,k} \subset U_{j-m}$, then
  \begin{align*}
    r_j^{-1} \norm{u}_{L^1(U_j)} 
    &\lesssim \sum_{k \in K_j} 
      r_j^{-1} \norm{u}_{L^1(B_{j,k})} 
      \lesssim \sum_{k \in K_j} 
      \norm{\opA u}_{L^1(\Omega \cap \lambda B_{j,k})} 
      \lesssim\abs{\opA u}(U_{j-m}).
  \end{align*}
  Overall, we obtain
  \begin{align*}
    \abs{\opA((1-\rho_j)u-u)}(\Omega) 
    &\leq \abs{\opA u}(U_{j-m}).
  \end{align*}
  Now, $\abs{\opA u}(U_{j-m}) \to 0$, since $U_{j-m} \searrow
  \emptyset$. This proves the claim by the Poincar\'e-inequality from Theorem \ref{thm:poincare2}.
\end{proof}

\begin{theorem}[Zero Traces]
  \label{thm:WA10}
  Let $\Omega \compactsubset U \subset \Rn$ for some open, bounded~$U$
  with Lipschitz boundary and let $u \in \WA(\Omega)$. The following
  are equivalent:
  \begin{enumerate}
  \item \label{itm:WA10-trace} $u \in \WA_0(\Omega)$.
  \item \label{itm:WA10-ext} The extension~$\tilde{u}:= \chi_\Omega u$ by
    zero on~$U \setminus \Omega$ is in $\WA(U)$.
  \item \label{itm:WA10-dense} There exist
    $u_k \in C^\infty_c(\Omega)$ with $u_k \to u$ in
    $\WA(\Omega)$.
  \end{enumerate}
\end{theorem}
\begin{proof}
  \ref{itm:WA10-trace} $\Rightarrow$ \ref{itm:WA10-ext}: Let
  $u \in \WA_0(\Omega)$ and let~$\tilde{u}=\chi_\Omega u$ be its zero
  extension on~$U$. Then by Corollary~\ref{cor:gluing} we have
  $\opA \tilde{u} = \opA u \mres_\Omega \in L^1(U)$, so $\tilde{u}
  \in \WA(U)$.

  \ref{itm:WA10-ext} $\Rightarrow$ \ref{itm:WA10-trace}: Let
  $\tilde{u} = \chi_\Omega u\in \WA(U)$. Then by
  Corollary~\ref{cor:gluing} we have
  $\opA \tilde{u} = \opA u \mres_\Omega + \trace^+(u) \otimes_\opA \nu
  \mathcal{H}^{n-1}\mres_{\partial \Omega}$. Since $\opA \tilde{u} \in L^1(U)$, the
  singular part must vanish, i.e. $\trace^+(u) \otimes_\opA \nu
  \mathcal{H}^{n-1}\mres_{\partial \Omega}=0$. So by $\setR$-ellipticity of $\opA$ we have $\trace^+(u)=0$
  on~$\partial \Omega$.
                      
  \ref{itm:WA10-dense} $\Rightarrow$ \ref{itm:WA10-trace}: By
  continuity of the trace operator we have $\trace(u)  = \lim_{k \to \infty}
  \trace(u_k)=0$ in $L^1(\partial \Omega)$, so $u \in \WA_0(\Omega)$.

  \ref{itm:WA10-trace} $\Rightarrow$ \ref{itm:WA10-dense}: Let
  $v_k := (1-\rho_k) u$ as in Lemma~\ref{lem:BV-cut}. Then $v_k \to u$
  in $\WA(\Omega)$. Moreover, the $v_k$ have compact support, since
  $v_k = 0$ on $U_{k+1}$.   Now, let $\eta_\epsilon \colon\R^{n}\to\R$
  be an standard mollifier with support on $B_\epsilon(0)$. Then we
  find~$\epsilon_k$ such that
  \begin{align*}
    \norm{v_k- v_k * \phi_{\epsilon_k}}_{L^1(\Omega)} +
    \norm{\opA v_k- \opA(v_k * \phi_{\epsilon_k})}_{L^1(\Omega)} &\leq 2^{-k}
  \end{align*}
  and $\support(v_k * \phi_{\epsilon_k}) \compactsubset \Omega$.  The
  sequence $u_k := v_k * \phi_{\epsilon_k}$ has the desired
  properties.
\end{proof}

\begin{proposition}[Trace--Preserving Area-Strict
  Smoothing]
  \label{prop:areastrictsmooth}
  Let $\Omega \compactsubset U \subset \Rn$ such that $\Omega$ and $U$
  are open and bounded and have Lipschitz boundary. Let $u_0 \in
  \WA(U)$. Futher let $u \in \BVA(U)$ with $u=u_0$ on $U \setminus
  \Omega$. Then  there exists $u_j \in u_0 + C^\infty_0(\Omega)$ such
  that $u_j \areato u$ in $\BVA(U)$.
\end{proposition}
\begin{proof}
  The proof is a straightforward modification of the corresponding
  statement for $\bv$-functions, see~\cite[Lemma B.2]{Bil03} or \cite[Lemma~1]{KrRiYM}. Let us
  just explain the basic idea: The usual localization argument by a
  partition of unity reduces the question to a local Lipschitz
  graph. Then split $u$ into $u_0 + \chi_\Omega (u-u_0)$. Now the
  $\chi_\Omega (u-u_0)$ part is moved by translation slightly
  into~$\Omega$. In a second step it is mollified to get
  a~$C^\infty_c(\Omega)$ term.
\end{proof}

\section{The Dirichlet Problem on $\bvA$--Spaces}
\label{sec:variational-problems}

This final section is devoted to variational problems with linear
growth involving~$\opA u$ subject to given boundary data.


Let $\Omega \subset \Rn$ be an open, bounded set with Lip\-schitz
boundary. Our goal is to study the
functional~$\mathfrak{F}\,:\, \WA(\Omega) \to \setR$ given by
\begin{align}
  \label{eq:functional}
  \mathfrak{F}[v]:=\int_{\Omega}f(x,\A v)\dif x.
\end{align}
subject to linear growth conditions. Given a boundary datum~$u_0 \in
\WA(\Omega)$, we wish to minimise~$\mathfrak{F}$ within the Dirichlet
class $u_0 + \WA_0(\Omega)$. The existence of a minimiser together
with the precise formulation of the problem at our disposal will be
given in Theorem~\ref{thm:main2} below.

Let us define the \emph{$\opA$-rank one cone}
$\CA = \RN \otimes_\opA \Rn \subset \RK$ with $\otimes_{\A}$ as given by \eqref{eq:bilinear}.  This cone is important to
characterise the jump terms of~$\BVA$ functions as in
Corollary~\ref{cor:gluing}. Also in the product
rule~\eqref{eq:productrule}, we have
$v \otimes_\opA \nabla \phi \in \CA$ pointwise for $\phi \in C^1(\Rn)$
and $v \in C^1(\Rn; \RN)$.

By use of the Fourier transform, we see that
$\opA(u) = (\opA[\xi] \hat{u})^\vee$. Since
$\opA[\xi] \hat{u} \in \CA$ pointwise, we obtain
$\opA(u) \in \linearspan(\CA)$ pointwise. Hence, we define the
\emph{effective range of~$\opA$} as
~$\RA := \linearspan(\CA) \subset \RK$, i.e.,  $\opA u \in \RA$
pointwise.  As a consequence, we only need to require that the second argument
of~$f$ in~\eqref{eq:functional} is from ~$\RA$. We assume that
\begin{align}
  \label{eq:f-cont}
  f\,:\, \overline{\Omega} \times \RA \to \setR \qquad \text{is continuous}
\end{align}
and satisfies the following linear growth assumption
\begin{align}
  \label{eq:lineargrowth}
  c_1|z|\leq f(x,z)\leq c_2|z|+c_3
\end{align}
for all $x \in \Omega$ and~$z \in \RA$. Moreover, we require~$\opA$ to
be $\setC$-elliptic, which allows us to use the trace results of the
previous sections.

Furthermore, we assume that there exists a
modulus of continuity $\omega$ such that 
\begin{align}
  \label{eq:f-modulus}
  |f(x,A)-f(y,A)|\leq \omega(|x-y|)(1+|A|)
\end{align}
holds for all $x,y\in\overline{\Omega}$ and all $A\in\RA$. In all of what follows, we tacitly stick to these assumptions.

We say that $g\,:\, \RA \to \setR$ is \emph{$\A$--quasiconvex} if for
all $\varphi\in\sobo_{0}^{1,\infty}((0,1)^{n};\RN)$ and $A\in \RA$
there holds
\begin{align}\label{eq:AQC}
  g(A)\leq \int_{(0,1)^{n}}g(A+\A\varphi)\dif x.
\end{align}
We say that~$f\,:\, \overline{\Omega} \times \RA \to \setR$ is
$\opA$--quasiconvex if~$f(x,\cdot)$ is $\opA$-quasiconvex for
each~$x \in \overline{\Omega}$.

Let us link this notion of quasiconvexity to that of Fonseca and M\"uller~\cite[Def.~3.1]{FonMul99}. Since
$\A$ is $\setC$-elliptic, it is also $\setR$-elliptic. So
by~\cite[Proposition~4.2]{Van13},  there exists
$M\in\mathbb{N}$ and a linear,
homogeneous constant coefficient differential operator $\mathbb{L}$
with symbol mapping $\mathbb{L}[\xi]$ from $\RK$ to $\R^{M}$ that
\emph{annihilates}~$\opA$ in the sense that the corresponding symbol
complex
\begin{align}
  \label{eq:complex}
  \RN \xrightarrow{\;\;\opA[\xi]\;} \RK
  \xrightarrow{\;\;\mathbb{L}[\xi]\;} \setR^M
\end{align}
is exact for every $\xi\in\R^{n}\setminus\{0\}$. In this situation,
$\A$ is called a \emph{potential} for $\mathbb{L}$, and $\mathbb{L}$
an \emph{annihilator} for $\mathbb{A}$. Since $\opA[\xi](\R^N)$
has the same dimension for all~$\xi\neq 0$, the operator~$\mathbb{L}$
has constant rank.  Consequently, our $\A$--quasiconvexity equals
$\mathbb{L}$--quasiconvexity\footnote{In \cite{FonMul99}, first order annihilating operators are considered, and in general this is not the case in our situation (e.g., the symmetric gradient is annihilated by $\curl\,\curl$). However, the generalisation of the concept of $\mathbb{L}$--quasiconvexity extends to higher order operators $\mathbb{L}$ in the obvious manner.} of Fonseca and M\"uller~\cite{FonMul99}.
By exactness of the above symbol complex \eqref{eq:complex}, it is easy to see that the
\emph{wave cone} (or characteristic cone)
$\Lambda_{\mathbb{L}}:=\bigcup_{\xi\in\R^{n}\setminus\{0\}}\ker(\mathbb{L}[\xi])$
of $\mathbb{L}$ agrees with our $\opA$-rank one cone~$\CA$.

We define the \emph{strong recession function}
$f^\infty \,:\, \overline{\Omega} \times \RA \to \setR$ by
\begin{align}
  \label{eq:defrecession}
  f^{\infty}(x,A):=\lim_{\substack{x'\to x\\ A'\to A\\
  t\to\infty}}\frac{f(x',tA')}{t},
\end{align}
whenever the limit exists. 

Since $f$ is $\opA$-quasiconvex, satisfies
the linear growth condition~\eqref{eq:lineargrowth} and the continuity
condition~\eqref{eq:f-modulus}, Lemma~\ref{lem:f-recession} from the appendix yields that
$f^\infty$ is automatically well-defined
on~$\overline{\Omega} \times \CA$. 

As usual the Dirichlet class $u_0 + \WA_0(\Omega)$ is not large enough
to ensure the existence of minimisers for variational problems with linear
growth. Here, the passage to~$\BVA(\Omega)$ allows to access the necessary
sequential compactness. However, elements of $\BVA(\Omega)$ do not admit
control over their \emph{exterior} trace. To overcome this problem we
proceed as in~\cite{GiaModSou79} and pass to a larger superset~$U$,
i.e., let $\Omega \compactsubset U$ with $\partial U$ Lipschitz. Now,
we extend~$\mathfrak{F}$ to $\BVA(U)$ and minimise over
those~$u \in \BVA(U)$ which agree with~$u_0$ on~$U \setminus \Omega$.
For this, we further need to accomplish the following: First, we have to
extend~$f\,:\, \overline{\Omega} \times \RA \to \setR$
to~$f\,:\, \overline{U} \times \RA \to \setR$, while preserving the
structure of~$f$, see Lemma~\ref{lem:f-recession2} in the
appendix. Second, we need to extend our boundary data to~$U$, which is
always possible, since
$\trace(\WA(\Omega))= L^1(\partial \Omega, \mathcal{H}^{n-1})=
\trace(W^{1,1}(U\setminus \Omega))$
by Theorem~\ref{thm:main1b}.  In particular, we assume in the
following that~$u_0 \in \WA(U)$.

We define the functional $\overline{\mathfrak{F}}_U\,:\, \BVA(U) \to \setR$ by
\begin{align*}
  \overline{\mathfrak{F}}_U[w]:=\int_{U}f \bigg(x,\frac{\dif\A
  w}{\dif\mathscr{L}^{n}} \bigg)\dif x + \int_{U}
  f^{\infty}\bigg(x,\frac{\dif\A w}{\dif  |\A^{s}w|}\bigg)\dif
  |\A^{s}w|
\end{align*}
and the Dirichlet class
\begin{align*}
  \mathcal{D}_{u_0} &= \set{w \in \BVA(U) \,:\, w=u_0 \text{ on } U
                      \setminus \overline{\Omega}}. 
\end{align*}
Hence, our aim is to minimise~$\mathfrak{F}_U$ over $\mathcal{D}_{u_0}$. Later we will see that this minimisation can also be expressed only in
terms of~$\BVA(\Omega)$ with an additional term $f^\infty(\cdot, \trace
(u-u_0) \otimes_\opA \nu)$ which penalises the deviations from the correct boundary values, see
Theorem~\ref{thm:main2}. 


We begin with a characterisation of the extension of
$\mathfrak{F}\,:\, \WA(\Omega) \to \setR$ to $\BVA(\Omega)$. For this, recall that $\Omega\subset\R^{n}$ is a bounded Lipschitz domain and that \eqref{eq:f-cont}--\eqref{eq:AQC} are in action.
\begin{proposition}
  \label{prop:Fbar-extension}
  Let $\overline{\mathfrak{F}}\,:\, \BVA(\Omega) \to \setR$ be given by
  \begin{align*}
    \overline{\mathfrak{F}}[u]:=\int_{\Omega}f \bigg(x,\frac{\dif\A
    u}{\dif\mathscr{L}^{n}} \bigg)\dif x + \int_{\Omega}
    f^{\infty}\bigg(x,\frac{\dif\A u}{\dif  |\A^{s}u|}\bigg)\dif
    |\A^{s}u|,
  \end{align*}
  is the $\A$-area strict continuous extension of
  $\mathfrak{F}\,:\, \WA(\Omega) \to \setR$. Moreover,
  $\overline{\mathfrak{F}}[u]\,:\, \BVA(\Omega) \to \setR$ is
  sequentially weak*--lower semicontinuous on $\BVA(\Omega)$. 
\end{proposition}
\begin{proof}
  We begin with the $\A$-area strict continuity
  of~$\overline{\mathfrak{F}}\,:\, \BVA(\Omega) \to
  \setR$. If~$f^\infty$ existed on all of~$\overline{\Omega}
  \times \RA$, we could just use~\cite[Theorem~4]{KriRin10}. However, 
  we can only rely on the existence of~$f^\infty$ on
  $\overline{\Omega} \times \CA$ due to Lemma~\ref{lem:f-recession}
  from the appendix. The following steps show how to overcome this
  technical issue and hence how the argument of \cite[Theorem~4]{KriRin10} can be made work. 

  Let us denote by $E(\overline{\Omega},\RA)$ those
  functions~$g\,:\, \overline{\Omega} \times \RA \to \setR$ such that
  $(x,\xi)\mapsto (1-|\xi|)g(x,(1-|\xi|)^{-1}\xi)$ has a continuous
  extension to $\overline{\Omega\times\mathbb{B}_{K}}$; here,
  $\mathbb{B}_{K}$ denotes the unit ball in $\RA$. In particular, the
  strong recession function~$g^\infty$ exists on all of
  $\overline{\Omega} \times \RA$. Functionals with integrands
  from~$E(\overline{\Omega},\RA)$ enjoy good continuity properties.

  Due to~\cite[Lemma~2.3]{AliBou97} there exists a sequence
  $f_k \in E(\overline{\Omega},\RA)$ with
  \begin{align}
    \sup_{k\in\mathbb{N}}f_k(x,A)=f(x,A)\;\;\; \text{and}\;\;\;
    \sup_{k\in\mathbb{N}}f_k^{\infty}(x,A)=f_{\#}(x,A) :=
    \liminf_{\substack{x'\to x\\ A'\to 
    A\\t\to \infty}}\frac{f(x',tA')}{t}.  
  \end{align}
  Let $u_j \areato u$ in $\BVA(\Omega)$. Since $f_k \in
  E(\overline{\Omega},\RA)$ we may apply the Reshetnyak--type continuity theorem in the version of~\cite[Theorem~5]{KriRin10} to conclude
  \begin{align*}
    \liminf_{j\to\infty}\overline{\mathfrak{F}}[u_{j}] 
    & \geq
      \liminf_{j\to\infty}\int_{\Omega}f_k\Big(x,\frac{\dif\A
      u_{j}}{\dif\mathscr{L}^{n}}\Big)\dif
      x +
      \int_{\Omega}f_k^{\infty}\Big(x,\frac{\dif\A^{s}u_{j}}{\dif
      |\A^{s}u_{j}|}\Big)\dif
      |\A^{s}u_{j}|
    \\
    & = \int_{\Omega}f_k\Big(x,\frac{\dif\A
      u}{\dif\mathscr{L}^{n}}\Big)\dif x +
      \int_{\Omega}f_k^{\infty}\Big(x,\frac{\dif\A^{s}u}{\dif
      |\A^{s}u|}\Big)\dif |\A^{s}u| 
  \end{align*}
  and so, by monotone convergence, 
  \begin{align*}
    \int_{\Omega}f\Big(x,\frac{\dif\A u}{\dif\mathscr{L}^{n}}\Big)\dif
    x + \int_{\Omega}f_{\#}\Big(x,\frac{\dif\A u}{\dif |\A^{s}u|}
    \Big)\dif |\A^{s}u|  \leq\liminf_{j\to\infty}\overline{\mathfrak{F}}[u_{j}].
  \end{align*}
  Due to the generalisation of Alberti's celebrated
  Rank--One Theorem by De~Philippis and Rindler in~\cite{DePRin16}, we
  know that $\frac{\dif\A u}{\dif |\A^{s}u|} \in \CA$ pointwisely
  $\abs{\opA^s u}$-a.e.\,. Now, by Lemma~\ref{lem:f-recession} from
  the appendix, we find that $f_\# = f^\infty$ on
  $\overline{\Omega} \times \CA$. Hence
  \begin{align*}
    \overline{\mathfrak{F}}[u] 
    &=  \int_{\Omega}f\Big(x,\frac{\dif\A u}{\dif\mathscr{L}^{n}}\Big)\dif
      x + \int_{\Omega}f^\infty\Big(x,\frac{\dif\A u}{\dif |\A^{s}u|}
      \Big)\dif |\A^{s}u|  \leq\liminf_{j\to\infty}\overline{\mathfrak{F}}[u_{j}].
  \end{align*}
  Since~$f$ is continuous, we may apply the same argument to~$-f$ to
  obtain
  $ \overline{\mathfrak{F}}[u] \geq
  \limsup_{j\to\infty}\overline{\mathfrak{F}}[u_{j}]$.
  Hence
  $\overline{\mathfrak{F}}[u] =
  \lim_{j\to\infty}\overline{\mathfrak{F}}[u_{j}]$.
  This proves that $\overline{\mathfrak{F}}\,:\, \BVA(\Omega) \to
  \setR$ is $\opA$--area strictly continuous.

  Due to Theorem~\ref{lem:approx-COmega-bar}, $\WA(\Omega)$ is dense
  in $\BVA(\Omega)$ with respect to $\opA$--area strict
  convergence. Since $\overline{\mathfrak{F}}=\mathfrak{F}$ on
  $\WA(\Omega)$, we see that $\overline{\mathfrak{F}}\,:\,
  \BVA(\Omega) \to \setR$ is the $\opA$-area strict extension
  of~$\mathfrak{F}\,:\, \WA(\Omega) \to \setR$.

  It remains to prove the sequential weak*--lower semicontinuity of
  $\overline{\mathfrak{F}}\,:\, \BVA(\Omega) \to \setR$ on
  $\BVA(\Omega)$. Let $\mathbb{L}$ be an $\opA$--annihilating operator
  as in the exact sequence~\eqref{eq:complex}. Now, the sequential
  weak*--lower semicontinuity just follows
  from~\cite[Theorem~1.2]{ArrDePRin17} (note that $f^\infty$ is well defined
  on~$\overline{\Omega} \times \CA$ due to Lemma~\ref{lem:f-recession}
  from the appendix). The proof is complete. 
\end{proof}
If we apply to our Dirichlet class $\mathcal{D}_{u_0}$, then we obtain
the following results:
\begin{corollary}
  \label{cor:Fu0-lsc}
  Let $f$ satisfy~\eqref{eq:f-cont}--\eqref{eq:AQC} and let
  $\overline{\mathfrak{F}}_{u_{0}}\,:\, \BVA(\Omega) \to \setR$ be given
  by
  \begin{align}\label{eq:relaxedF}
    \begin{split}
      \overline{\mathfrak{F}}_{u_{0}}[u]
      &:=\int_{\Omega}f\Big(x,\frac{\dif\A
        u}{\dif\mathscr{L}^{n}}\Big)\dif\mathscr{L}^{n} +\int_{\Omega}
      f^{\infty}\Big(x,\frac{\dif \A u}{\dif |\A^{s}u|}\Big)\dif
      |\A^{s}u|
      \\
      & \;+ \int_{\partial\Omega}f^{\infty}\Big(x,\nu_{\partial\Omega}
      \otimes_{\A}\trace(u-u_{0})\Big) \dif\mathcal{H}^{n-1}
    \end{split}
  \end{align}
  is sequentially weak*--lower semicontinuous on~$\BVA(\Omega)$.
\end{corollary}
\begin{proof}
  Proposition~\ref{prop:Fbar-extension} (applied with $\Omega$ replaced by $U$) shows that
  $\overline{\mathfrak{F}}_U\,:\, \BVA(U) \to \setR$ is area-strictly
  continuous on $\BVA(U)$ and sequentially weak*--lower semicontinuous
  on~$\BVA(U)$.

  For $u \in \BVA(\Omega)$ let
  $\widetilde{u} := \chi_{U \setminus \overline{\Omega}} u_0 +
  \chi_\Omega u$.
  Then due to Corollary~\ref{cor:gluing} we
  have~$\widetilde{u} \in \BVA(U)$ and, with the outer normal~$\nu$
  of~$\Omega$,
  \begin{align}
    \label{eq:decomp-vtilde}
    \A \widetilde{u} =\A u\mres\Omega + \A u_{0}\mathscr{L}^{n}\mres
    (U\setminus\overline{\Omega}) +
    \trace(u-u_{0})\otimes_{\A}\nu
    \mathcal{H}^{n-1}\mres\partial\Omega.  
  \end{align}
  Hence,
  \begin{align}
    \label{eq:addconst}
    \overline{\mathfrak{F}}_{U}[\widetilde{u}] = \overline{\mathfrak{F}}_{u_{0}}[u] +
    \int_{U\setminus\overline{\Omega}}f(x,\opA
    u_{0})\dif
    x.
  \end{align}
  If $u_k \weakastto u$ in $\BVA(\Omega)$, then
  $\widetilde{u_k} \weakastto \widetilde{u}$ in $\BVA(U)$. Indeed, it
  is clear that $u_k \to u$ in $L^1(U)$. Moreover, since $u_k$ is
  bounded in $\BVA(\Omega)$, so is $\opA u_k \in \mathcal{M}(\Omega)$
  and $\trace(u_k)$ in $L^1(\partial \Omega)$ (using the Trace
  Theorem~\ref{thm:main1b}). This and~\eqref{eq:decomp-vtilde} shows
  that $\widetilde{u_k}$ is bounded in~$\BVA(U)$. In conjunction with
  $u_k \to u$ in $L^1(U)$ we obtain
  $\widetilde{u_k} \weakastto \widetilde{u}$ in $\BVA(U)$.

  Since~$\overline{\mathfrak{F}}_U$ is sequentially weak*--lower
  semicontinuous on~$\BVA(U)$, it follows that
  $\overline{\mathfrak{F}}_{u_0}$ sequentially weak*--lower semicontinuous
  on~$\BVA(\Omega)$.
\end{proof}

\begin{theorem}
  \label{thm:main2}

  Let $f$ satisfy~\eqref{eq:f-cont}--\eqref{eq:AQC}. Then the
  functional
  $\overline{\mathfrak{F}}_{u_{0}}\,:\, \BVA(\Omega) \to \setR$ is
  coercive and has a minimiser on~$\BVA(\Omega)$. Moreover, we have 
  \begin{align}
    \label{eq:consistency}
    \min_{\bvA(\Omega)}\overline{\mathfrak{F}}_{u_{0}}=
    \inf_{u_0+ \WA_0(\Omega)}\mathfrak{F}.  
  \end{align}
\end{theorem}
\begin{proof}
  We begin with the coerciveness of $\overline{\mathfrak{F}}_{u_{0}}$.
  Let $(v_k) \subset\BVA(\Omega)$ with $(\overline{\mathfrak{F}_{u_0}}(u_k))$
  bounded. We have to show that $(v_k)$ is bounded in
  $\BVA(\Omega)$. Let
  $\widetilde{v_k}:= \chi_{U \setminus \overline{\Omega}} u_0 +
  \chi_\Omega v_k$
  as in Corollary~\ref{cor:Fu0-lsc}. Then due to~\eqref{eq:addconst},
  $\overline{\mathfrak{F}}_U(\widetilde{v_k})$ is bounded.  By the
  linear growth condition~\eqref{eq:lineargrowth} we see that
  $(\A v_k)$ is uniformly bounded in $\mathcal{M}(U;\RK)$.  Now choose
  a ball $B' \subset \Omega$ and another ball~$B$ with $U \subset B$.
  Since $v_k -u_0 = 0$ on $U \setminus \overline{\Omega}$, we can
  extend it by zero to a function from~$\BVA(B)$ due to Theorem \ref{thm:WA10} (b). Now, we can apply
  \Poincare's inequality in the form of Theorem~\ref{thm:poincare2} to
  conclude that $(v_k)$ is also bounded in $L^1(U)$. Hence, $(v_k)$ is
  bounded on~$\BVA(\Omega)$, which is the desired coerciveness.

  By positivity of $f$ and $f^\infty$,
  $\overline{\mathfrak{F}}_{u_{0}}[w]\geq 0$ for all
  $w\in\bvA(\Omega)$, and so we may pick a minimising sequence $(u_k)$
  in $\BVA(\Omega)$. By coerciveness, this sequence is bounded
  in~$\BVA(\Omega)$. We can pick a (non--relabeled) subsequence such
  that $u_k \weakastto u$ in $\BVA(\Omega)$ for some
  $u\in \BVA(\Omega)$. By the sequential weak*--lower semicontinuity from
  Corollary~\ref{cor:Fu0-lsc}, we deduce that~$u$ is a minimiser
  of~$\overline{\mathfrak{F}}_{u_0}$. 

  We conclude the proof by showing~\eqref{eq:consistency}. The
  '$\leq$'-part is obvious. Due to
  Proposition~\ref{prop:areastrictsmooth} we find a sequence
  $w_k \in \mathcal{D}_{u_0}$ such that $w_k \areato u$ in
  $\BVA(U)$. By the $\opA$--area strict continuity
  of~$\overline{\mathfrak{F}}_U$ on $\BVA(U)$, see
  Proposition~\ref{prop:Fbar-extension}, we see that
  $\overline{\mathfrak{F}}_U(u) = \lim_{k \to \infty}
  \overline{\mathfrak{F}}_U(w_k)$.
  This and~\eqref{eq:addconst} proves the '$\geq$'-part
  of~\eqref{eq:consistency}.
\end{proof}

\section{Appendix}
\label{sec:appendix}
We now collect some auxiliary results that have been used in the main part of the paper. 
The following lemma shows that the recession function is automatically
well-defined on the $\opA$-rank one cone.
\begin{lemma}
  \label{lem:f-recession}
  Let $\opA$ be $\setR$-elliptic, let
  $f\,:\, \overline{\Omega} \times \RA \to \setR$ be
  $\opA$-quasiconvex in the sense of~\eqref{eq:AQC}, satisfy the linear growth
  condition~\eqref{eq:lineargrowth} and the continuity
  condition~\eqref{eq:f-modulus}. Then $f(x,\cdot)$ is Lipschitz continuous
  in~$\RA$ uniformly in~$x \in \overline{\Omega}$. Moreover, the
  strong recession function
  $f^\infty\,:\, \overline{\Omega} \times \RA \to \setR$ with
  \begin{align*}
    f^{\infty}(x,A):=\lim_{\substack{x'\to x\\ A'\to A\\
    t\to\infty}}\frac{f(x',tA')}{t}
  \end{align*}
  is well-defined on~$\overline{\Omega} \times \CA$.  (Note that the
  limit $A'\to A$ is taken in~$\RA$.) Moreover,
  \begin{align*}
    \abs{f^\infty(x,A) - f^\infty(x',A)} &\leq \omega(\abs{x'-x}) \abs{A}
  \end{align*}
  for all $x, x' \in \overline{\Omega}$ and $A \in \CA$.
\end{lemma}
\begin{proof}
  We begin with the Lipschitz continuity of~$f$ on $\RA$.

  Let $A \in \RA$ and $B = a \otimes_{\opA} b\in \CA$.  Since $f$ is $\opA$-quasiconvex,
  it is a consequence\footnote{As proven in~\cite{FonMul99}, if
    $\mathcal{A}$ is a first order linear homogeneous differential
    operator, then $\mathcal{A}$--quasiconvex functions are
    $\Lambda_{\mathcal{A}}$--convex. Note that in our setting,
    $\mathbb{L}=\mathcal{A}$ need not be first of first order,
    however, their arguments extend to the case of higher order
    annihilating operators $\mathbb{A}$ in a straightforward manner.}
  of~\cite[Prop. 3.4]{FonMul99} that $t \mapsto f(x,A+tB)$ is convex
  on~$\setR$. This property is known as $\CA$-convexity,
  see~\cite{KiKr16}.

  Thus the function $g(t):=\abs{f(x,A+ta\otimes_{\A}b)-f(x,A)}/t$ is
  increasing. Hence, with $\lambda:=(1+|A+B|+|A|)/|B|>1$, we obtain
  \begin{align*}
    \abs{f(x,A+B)-f(x,A)}
    & =g(1) \leq g(\lambda)
    \\
    &\leq |f(x,A+\lambda
      a\otimes_{\A}b)-f(x,A)|\frac{|B|}{1+|A+B|+|A|}
    \\
    & \leq \frac{c_2(2\abs{A}+\lambda\abs{B})+ 2c_3}{1+|A+B|+|A|}|B| 
    \\
    & \leq \frac{c_2(1+3|A|+\abs{A+B})+ 2c_3}{1+|A+B|+|A|}|B| 
    \\
    &\leq (3c_2+2c_3) |B|
  \end{align*}
using~\eqref{eq:lineargrowth}.
  This proves the Lipschitz continuity in $\CA$-directions.

  If $B \in \RA$, then by $\RA=\linearspan(\CA)$ we can decompose~$B$
  into at most~$K$ summands from~$\CA$. Now the Lipschitz continuity
  in~$\CA$-directions, implies
  \begin{align}
    \label{eq:f-lip-RA}
    \abs{f(x,A+B)-f(x,A)} &\leq K(3c_2+2c_3) |B|
  \end{align}
  for all $A,B \in \RA$. This proves the Lipschitz continuity part.

  Let $A \in \CA$ and $x \in \overline{\Omega}$. Then
  $t \mapsto (f(x,tA)-f(x,0))/t$ is increasing in~$t$ by
  $\CA$-convexity of~$f(x,\cdot)$ and bounded by~$c_2\abs{A}$due toy the
  linear growth condition~\eqref{eq:lineargrowth}. This allows us to
  define $g^\infty\,:\, \overline{\Omega} \times \CA \to \setR$ by
  \begin{align*}
    g^\infty(x,A)=\lim_{t\to\infty}\frac{f(x,tA)}{t} =
    \sup_{t>0}\frac{f(x,tA)}{t}. 
  \end{align*}

  Now, let $A' \in \RA$ and $x' \in \overline{\Omega}$, then
  with~\eqref{eq:f-lip-RA} and~\eqref{eq:f-modulus}
  \begin{align*}
    \biggabs{\frac{f(x',tA')}{t} - \frac{f(x,tA)}{t}} 
    &\leq
      \biggabs{
      \frac{f(x',tA')-f(x',tA)}{t}}
      +
      \biggabs{\frac{f(x',t 
      A)-f(x,tA)}{t}}   
      \\
    &\leq K(3c_2+2c_3) \abs{A-A'} + \omega(\abs{x'-x}) \frac{1+ t \abs{A}}{t}.
  \end{align*}
  This proves $f^\infty(x,A) = g^\infty(x,A)$ for all
  $x \in \overline{\Omega}$ and $A \in \CA$. Consequently, we obtain the existence
  of~$f^\infty$ in $\overline{\Omega} \times \CA$.

  The continuity of $f^\infty(\cdot,A)$ for $A \in \CA$ is a direct
  consequence of the continuity of~$f(\cdot,A)$.
\end{proof}
\begin{lemma}
  \label{lem:f-recession2}
  Let $\opA$ be $\setR$-elliptic, let
  $f\,:\, \overline{\Omega} \times \RA \to \setR$ be
  $\opA$-quasiconvex in the sense of~\eqref{eq:AQC}, satisfy the linear growth
  condition~\eqref{eq:lineargrowth} and the continuity
  condition~\eqref{eq:f-modulus}. Furthermore,
  let~$\Omega \compactsubset U$ with $\partial U$ Lipschitz. Then
  there exists an extension
  $\tilde{f}\,:\,\overline{U} \times \RA \to \setR$ of~$f$, which is
  $\opA$-quasiconvex, satisfies the linear growth
  condition~\eqref{eq:lineargrowth} and the continuity
  condition~\eqref{eq:f-modulus}. (The modulus of continuity might
  change by a factor.)
\end{lemma}
\begin{proof}
  Since~$\partial U$ and $\partial \Omega$ are Lipschitz, we find a
  Lipschitz map $\Phi\,:\, \overline{U} \to \overline{\Omega}$, which
  is the identity on~$\overline{\Omega}$. Now define
  $\tilde{f}(x,A) := f(\Phi(x),A)$. 
\end{proof}

\hspace*{1ex}

\centerline{\bf Declaration}
\noindent{
The authors declare that there are no conflicts of interest.}

\bibliographystyle{amsalpha}
\bibliography{traces}

\end{document}